\documentclass[12pt]{amsart}
\usepackage{amsmath, amssymb, graphicx}
\usepackage{bm}
\usepackage[margin=1in]{geometry}
\usepackage{nicefrac,mathrsfs}
\usepackage{float, mathtools}
\usepackage{hyperref, color, enumitem}

\newtheorem{theorem}{Theorem}[section]
\newtheorem{letteredtheorem}{Theorem}

\newtheorem{lemma}[theorem]{Lemma}
\newtheorem{corollary}[theorem]{Corollary}
\newtheorem{proposition}[theorem]{Proposition}

\newtheorem{conjecture}[theorem]{Conjecture}

\theoremstyle{definition}
\newtheorem{definition}[theorem]{Definition}
\newtheorem{example}[theorem]{Example}
\newtheorem{remark}[theorem]{Remark}

\newcommand{\newword}[1]{\emph{#1}}
\numberwithin{equation}{section}

\newtheoremstyle{named}{}{}{\itshape}{}{\bfseries}{.}{.5em}{#1 \thmnote{#3}}
\theoremstyle{named}
\newtheorem*{namedtheorem}{Theorem}

\newcommand{\R}{\mathbb{R}}

\def\F{\mathcal{F}}
\def\LD{\mathop{\mathrm{LD}}}
\def\RD{\mathop{\mathrm{RD}}}
\def\indeg{\mathop{\mathrm{indeg}}}
\def\indegG{\mathop{\mathrm{indeg}_G}}
\def\indegL{\mathop{\mathrm{indeg}_L}}
\def\indegHi{\mathop{\mathrm{indeg}_{H_i}}}
\def\indegKj{\mathop{\mathrm{indeg}_{K_j}}}
\def\outdeg{\mathop{\mathrm{outdeg}}}
\def\indegK{\mathop{\indeg_{K_{n+1}}}}
\def\outdegG{\mathop{\mathrm{outdeg_G}}}
\def\outdegF{\mathop{\mathrm{outdeg_F}}}

\def\aff{\mathop{\mathrm{aff}}}
\def\conv{\mathop{\mathrm{Conv}}}
\def\Tri{\mathop{\mathrm{Tri}}}
\def\Arr{\mathop{\mathrm{Arr}}}
\def\Sol{\mathop{\mathrm{Sol}}}
\def\inseq{\mathop{\mathrm{InSeq}}}
\def\gr{\mathop{\mathrm{Gr}}}
\def\poly{\mathop{\mathrm{Poly}}}

\def\agr{\mathop{\mathrm{Gr}^{\mathrm{aug}}}}
\def\Gaug{\mathop{G^{\mathrm{aug}}}}
\def\SolG{\mathop{\mathrm{Sol}_G}}
\def\SolGm{\mathop{\mathrm{Sol}_{G_m}}}

\makeatletter
\def\Ddots{\mathinner{\mkern1mu\raise\p@
		\vbox{\kern7\p@\hbox{.}}\mkern2mu
		\raise4\p@\hbox{.}\mkern2mu\raise7\p@\hbox{.}\mkern1mu}}
\makeatother
\font\co=lcircle10

\def\jr{\smash{\raise2pt\hbox{\co \rlap{\rlap{\char'005} \char'007}}
		\raise6pt\hbox{\rlap{\vrule height6.5pt}}
		\raise2pt\hbox{\rlap{\hskip4pt \vrule height0.4pt depth0pt
				width7.7pt}}}}
\def\je{\smash{\raise2pt\hbox{\co \rlap{\rlap{\char'005}
				\phantom{\char'007}}}\raise6pt\hbox{\rlap{\vrule height6pt}}}}
\def\+{\smash{\lower2pt\hbox{\rlap{\vrule height14pt}}
		\raise2pt\hbox{\rlap{\hskip-3pt \vrule height.4pt depth0pt
				width14.7pt}}}}
\def\textcross{\ \smash{\lower4pt\hbox{\rlap{\hskip4.15pt\vrule height14pt}}
		\raise2.8pt\hbox{\rlap{\hskip-3pt \vrule height.4pt depth0pt
				width14.7pt}}}\hskip12.7pt}
\def\textelbow{\ \hskip.1pt\smash{\raise2.8pt%
		\hbox{\co \hskip 4.15pt\rlap{\rlap{\char'005} \char'007}
			\lower6.8pt\rlap{\vrule height3.5pt}
			\raise3.6pt\rlap{\vrule height3.5pt}}
		\raise2.8pt\hbox{%
			\rlap{\hskip-7.15pt \vrule height.4pt depth0pt width3.5pt}%
			\rlap{\hskip4.05pt \vrule height.4pt depth0pt width3.5pt}}}
	\hskip8.7pt}
\title[Generalized permutahedra to Grothendieck polynomials via flow polytopes]{From generalized permutahedra to Grothendieck polynomials via flow polytopes}

\author{Karola M\'esz\'aros}
\address{Karola M\'esz\'aros, Department of Mathematics, Cornell University, Ithaca, NY 14853 and School of Mathematics, Institute for Advanced Study, Princeton, NJ 08540.  \newline\textup{karola@math.cornell.edu}
}

\author{Avery St.~Dizier}
\address{Avery St.~Dizier, Department of Mathematics, Cornell University, Ithaca NY 14853.  \newline{ajs624@cornell.edu}
}

\thanks{M\'esz\'aros was partially supported by National Science Foundation Grants (DMS 1501059 and DMS 1847284), as well as by a von Neumann Fellowship at the IAS funded by the Fund for Mathematics and Friends of the Institute for Advanced Study.}

\begin{document}
	
	\begin{abstract}
		We study a family of dissections of flow polytopes arising from the subdivision algebra. To each dissection of a  flow polytope, we associate a polynomial, called the left-degree polynomial, which we show is invariant of the dissection considered (proven independently by Grinberg). We prove that left-degree polynomials encode integer points of generalized permutahedra. Using that certain left-degree polynomials are related to Grothendieck polynomials, we resolve special cases of conjectures by Monical, Tokcan, and Yong regarding the saturated Newton polytope property of Grothendieck polynomials.
	\end{abstract}
	 
	\maketitle
	
	\section{Introduction}
	\label{sec1}	
	
	The flow polytope $\F_G$  associated to a directed acyclic graph $G$ is the set of all flows $f:E(G) \rightarrow \mathbb{R}_{\geq 0}$ of size one. Flow polytopes are   fundamental objects in combinatorial optimization \cite{schrijver}, and  in the past decade they were also uncovered in representation theory \cite{bv, mm}, the study of the space of diagonal harmonics \cite{lmm, tesler}, and the study of Schubert and Grothendieck polynomials \cite{toric,pipe1}.  A natural way to analyze a
	convex polytope is to dissect it into simplices. The relations of the subdivision algebra, developed in a series of papers \cite{root1, root2, prod}, encode dissections of a family of flow (and root) polytopes (see  Section \ref{sec2} for details).
	
	Take any graph $G$ with special source and sink vertices and fix a dissection $\mathcal{R}$ (into simplices) produced by the subdivision algebra. We study an invariant of $\mathcal{R}$ called the left-degree polynomial. Left-degree polynomials were introduced in \cite{pipe1} by Escobar and M\'esz\'aros. They showed that for a family of trees, the left-degree polynomial does not depend on the particular dissection considered. In Theorem \ref{theoremA}, we extend this result to any (not necessarily simple) graph. This was independently proven by Grinberg in \cite{grinberg} using algebraic techniques.
	
	Our main technique is to connect left-degree polynomials to flow polytopes. We study the left-degree polynomial of a particular recursive dissection from \cite{prod}. In Corollary 3.16, we partition the support of this left-degree polynomial (with multiplicity) into blocks and show that the convex hull of each block is integrally equivalent to a flow polytope. Using this flow perspective, we give an inductive proof of Theorem \ref{theoremA}.
	
	Using the flow approach again, we connect the Newton polytopes of left-degree polynomials to generalized permutahedra. In Theorem \ref{theoremB}, we show that the Newton polytope of every homogeneous component of a left-degree polynomial is a generalized permutahedron. We also prove the \emph{saturated Newton polytope} property (SNP) of Monical, Tokcan, and Yong \cite{MTY}: every integer point in the Newton polytope is in the support of the polynomial.
	
	We apply these results to Schubert and Grothendieck polynomials. Escobar and M\'esz\'aros showed in \cite[Theorem 5.3]{pipe1} that a certain family of Grothendieck polynomials are related to left-degree polynomials. We conclude in Theorem \ref{theoremC} that this family of Grothendieck polynomials have SNP, and that the Newton polytopes of their homogeneous components are generalized permutahedra. We conjecture this holds for all Grothendieck polynomials (Conjecture \ref{conj:groth}).
	
	The outline of this paper is as follows. Section \ref{sec2} covers the necessary background. In Section \ref{sec3}, we study the support of left-degree polynomials (left-degree sequences) directly, and make the connection to flow polytopes. To maximize ease of reading, we restrict to the case of simple graphs. In Section \ref{sec4} we introduce left-degree polynomials and describe their Newton polytopes. We apply this description to a family of Grothendieck polynomials in Section \ref{sec5}. In Section \ref{sec6}, we describe the technical modifications required to drop the simple graph assumption in the previous sections. We combinatorially prove left-degree polynomials are an invariant of the underlying graph.
	
	\section{Background information}
	\label{sec2}
	In this section, we summarize  definitions, notations, and results that we  use later. Throughout this paper, by \newword{graph}, we mean a directed acyclic graph where multiple edges are allowed (as described below). Although we sometimes refer to edges by their endpoints, we allow that $G$ may have multiple edges. We also adopt the convention of viewing each element of a multiset as being distinct, so that we may speak of subsets, though we will use the word submultiset interchangeably to highlight the multiplicity. Due to this convention, all unions in this paper are assumed to be disjoint multiset unions. For any integers $m$ and $n$, we will frequently use the notation $[m,n]$ to refer to the set $\{m,m+1,\ldots, n\}$ and $[n]$ to refer to the set $[1,n]$.
	
	\subsection{Flow Polytopes} Let $G$ be a graph on vertex set $[0,n]$ with edges directed from smaller to larger vertices. For each edge $e$, let $\mathrm{in}(e)$ denote the smaller (initial) vertex of $e$, and $\mathrm{fin}(e)$ the larger (final) vertex of $e$. Imagine fluid moving along the edges of $G$. At vertex $i$ let there be an external inflow of fluid $a_i$ (outflow of $-a_i$ if $a_i<0$), and call the vector $\bm{a}=(a_0,\ldots,a_n)\in\mathbb{R}^{n+1}$ the \newword{netflow vector}. Formally, a \newword{flow} on $G$ with netflow vector $\bm{a}$ is an assignment $f:E(G)\to \R_{\geq 0}$ of nonnegative values to each edge such that fluid is conserved at each vertex. That is, for each vertex $i$
	\[ \sum_{\mathrm{in}(e)=i}{f(e)} - \sum_{\mathrm{fin}(e)=i}{f(e)} = a_i.\]

	The \newword{flow polytope} $\mathcal{F}_G(\bm{a})$ is the collection of all flows on $G$ with netflow vector $\bm{a}$. Alternatively, let $M_G$ denote the incidence matrix of $G$. That is, let the columns of $M_G$ be the vectors $e_i-e_j$ for $(i,j)\in E(G)$, $i<j$,  where $e_i$ is the $(i+1)$-th standard basis vector in $\R^{n+1}$. Then,
	\begin{align}
		\label{def:flowpolytope}
		\mathcal{F}_G(\bm{a})= \{f\in \R^E_{\geq 0}\mid  M_G f=\bm{a} \}.
	\end{align}
	From this perspective, note that the number of integer points in $\mathcal{F}_G(\bm{a})$ is exactly the number of ways to write $\bm{a}$ as a nonnegative integral combination of the vectors $e_i-e_j$ for edges $(i,j)$ in $G$, $i<j$. This number is known as the \newword{Kostant partition function} $K_G(\bm{a})$. For brevity, we write $\F_G$ to mean $\F_G(1,0,\ldots , 0,-1)$, and we refer to $\F_G$ as the flow polytope of $G$, since in this paper our primary focus is on studying these particular flow polytopes.
		
	The following milestone result on volumes of flow polytopes was shown by Postnikov and Stanley in unpublished work.
	
	\begin{theorem}[Postnikov-Stanley]
		\label{pstheorem}
		Let $G$ be a directed acyclic connected graph on vertex set $[0,n]$. Set $d_i=\indeg_G(i)-1$ for each vertex $i$, where $\indeg_G(i)$ is the number of edges incoming to vertex $i$ in $G$. The normalized volume of the flow polytope of $G$ is given by
		\[ \mathrm{Vol} \,\,\mathcal{F}_{G} = K_{G} \left (0,d_1, \ldots , d_n, \, -\sum_{i=1}^{n}{d_i} \right ). \]
	\end{theorem}
	
	Baldoni and Vergne \cite{bv} generalized this result for  flow polytopes with arbitrary netflow vectors. Theorem \ref{pstheorem}  beautifully connects the volume of the flow polytope of any graph to an evaluation of the Kostant partition function. We note that since the number of integer points of a flow polytope is already given by a Kostant partition function evaluation, the volume of any flow polytope is given by the number of integer points of another.

	Recall that two polytopes $P_1\subseteq \R^{k_1}$ and ${P_2\subseteq \mathbb{R}^{k_2}}$ are \newword{integrally equivalent} if there is an affine transformation $T:\R^{k_1}\to \R^{k_2}$ that is a bijection $P_1\to P_2$ and a bijection $\aff(P_1)\cap \mathbb{Z}^{k_1}\to \aff(P_2)\cap \mathbb{Z}^{k_2}$. Integrally equivalent polytopes have the same face lattice, volume, and Ehrhart polynomial.
	
	 Given a graph $G$ and a set $S$ of its edges, we use the notation $G/S$ to denote the graph obtained from $G$ by contracting the edges in $S$ (and deleting loops). We use the notation $G\backslash S$ to denote the graph obtained from $G$ by deleting the edges in $S$. For a set $V$ of vertices of $G$, we also use the notation $G\backslash V$ to denote the graph obtained from $G$ by deleting the vertices in $V$ together with all edges incident to them. When $S$ or $V$ consists of just one element, we simply write $G/e$ or $G\backslash v$.
	 
	 While simple to prove, the following lemma is important. We leave its proof to the reader.

	\begin{lemma}
		\label{contraction}
		Let $G$ be a graph on $[0,n]$. Assume vertex $j$ has only one outgoing edge $e$ and netflow $a_j\geq0$. If $e$ is directed from $j$ to $k$, then 
		\[\mathcal{F}_G(a_0,\dots, a_n) \mbox{ and } \mathcal{F}_{G/e}(a_0,\ldots, a_{j-1}, a_{j+1},a_{j+2},\ldots, a_{k-1}, a_k+a_j,a_{k+1},\ldots, a_n)\] are integrally equivalent. An analogous result holds if $j$ has only one incoming edge and $a_j\leq 0$.
	\end{lemma}

	\subsection{Dissections of Flow Polytopes}
	For graphs with a special source and sink, there is a systematic way to dissect the flow polytope $\F_{\widetilde{G}}$ studied in \cite{prod}. Let $G$ be a graph on $[0,n]$, and define $\widetilde{G}$ on $[0,n]\cup \{s,t\}$ with $s$ being the smallest vertex and $t$ the biggest vertex by setting $E(\widetilde{G})= E(G)\cup \{(s,i),(i,t)\mid i\in[0,n] \}$. Although we defined the flow polytope $\mathcal{F}_G(\bm{a})$ above only when $G$ was a graph on $[0,n]$, the definition (\ref{def:flowpolytope}) makes sense with any totally ordered vertex set. For graphs $\widetilde{G}$, we take the ordering $s<0<1<\cdots<n<t$. The systematic dissections of $\mathcal{F}_{\widetilde{G}}$ can be expressed either in the language of the subdivision algebra or in terms of reduction trees \cite{root1, root2, prod}. We use the language of reduction trees.
	
	Let $G_0$ be a graph on $[0,n]$ with edges $(i,j)$ and $(j,k)$ for some $i<j<k$. By a \newword{reduction} on $G_0$, we   mean the construction of three new graphs $G_1$, $G_2$ and $G_3$ on $[0,n]$ given by
	\begin{align}
		E(G_1)&=E(G_0)\backslash \{(j,k)\}\cup \{(i,k)\}\nonumber \\
		E(G_2)&=E(G_0)\backslash \{(i,j)\}\cup \{(i,k)\} \label{reducing} \\
		E(G_3)&=E(G_0)\backslash \{(i,j),(j,k)\}\cup \{(i,k)\} \nonumber
	\end{align}
	See Figure \ref{subdivlemmaproofpic} for an example reduction. We say $G_0$ \newword{reduces} to $G_1$, $G_2$ and $G_3$. We also say that the above reduction is at vertex $j$, on the edges $(i,j)$ and $(j,k)$. The following proposition explains how the process of taking reductions dissects the flow polytope $\mathcal{F}_{G_0}$ into other flow polytopes.
	
	\begin{proposition}
		\label{subdivisionlemma} 
		Let $G_0$ be a graph on $[0,n]$ which reduces to $G_1$, $G_2$ and $G_3$ as above. Then for each $m\in[3]$, there is a polytope $Q_m$ integrally equivalent to $\mathcal{F}_{\widetilde{G}_m}$ such that $Q_1$ and $Q_2$ subdivide $\mathcal{F}_{\widetilde{G}_0}$ and intersect in $Q_3$. That is, the polytopes $Q_1$, $Q_2$, and $Q_3$ satisfy
		\[\mathcal{F}_{\widetilde{G}_0} = Q_1 \bigcup  Q_2 \mbox{ with }  Q_1^o\bigcap  Q_2^o= \emptyset \mbox{ and }  Q_1\bigcap  Q_2= Q_3.\]
		Moreover, $Q_1$ and $Q_2$ have the same dimension as $\mathcal{F}_{\widetilde{G}_0}$, and $Q_3$ has dimension one less.
	\end{proposition}

	\begin{proof}
		Let $r_1$ and $r_2$ denote the edges of $G_0$ from $i$ to $j$ and from $j$ to $k$ respectively that were used in the reduction. Viewing $\mathbb{R}^{\#E(\widetilde{G}_0)}$ as functions $f:E(\widetilde{G}_0)\to \mathbb{R}$, cut $\mathcal{F}_{\widetilde{G}_0}$ with the hyperplane $H$  defined by the equation $f(r_1)=f(r_2)$. Let $Q_1$ be the intersection of $\mathcal{F}_{\widetilde{G}_0}$ with the positive half-space $f(r_1)\geq f(r_2)$, let $Q_2$ be the intersection of $\mathcal{F}_{\widetilde{G}_0}$ with the negative half-space $f(r_1)\leq f(r_2)$, and let $Q_3$ be the intersection of $\mathcal{F}_{\widetilde{G}_0}$ with the hyperplane $H$.
		See  Figure \ref{subdivlemmaproofpic} for an  illustration of  the integral equivalence between $Q_m$ and $\mathcal{F}_{\widetilde{G}_m}$. Notice that since we are doing the reductions on the edges of  $G_0$ (as opposed to on the edges incident to the source or sink in $\widetilde{G}_0$), it follows that the hyperplane $H$ meets $\mathcal{F}_{\widetilde{G}_0}$ in its interior, giving the claims on the dimensions of each $Q_m$.	\end{proof}
		\begin{figure}[ht]
			\includegraphics[]{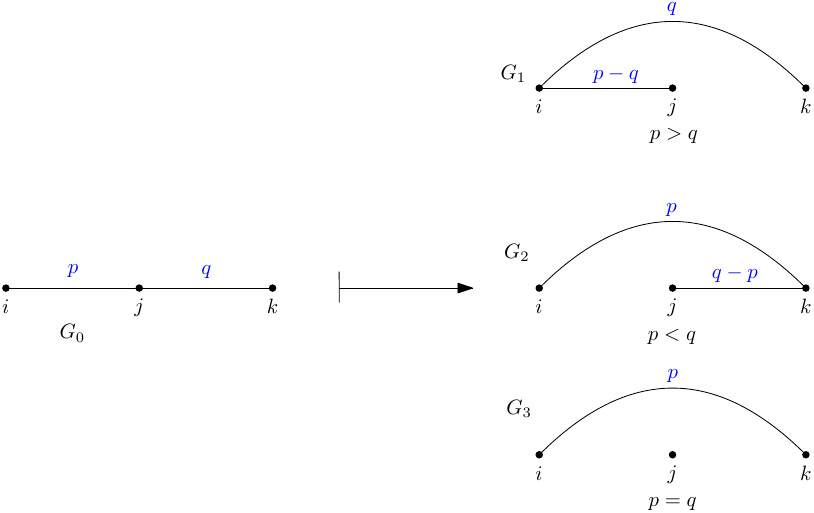}
			\caption{An illustration of the integral equivalence between $Q_m$ and $\mathcal{F}_{\widetilde{G}_m}$ for $m\in[3]$ used Proposition \ref{subdivisionlemma}.}
			\label{subdivlemmaproofpic}
		\end{figure}
	
	Iterating this subdivision process will produce a dissection of $\mathcal{F}_{\widetilde{G}_0}$ into simplices. This process can be encoded using a reduction tree. A \newword{reduction tree} of $G$ is constructed as follows. Let the root node of the tree be labeled by $G$. If a node has any children, then it has three children obtained by performing a reduction on that node and labeling the children with the graphs defined in (\ref{reducing}). Continue this process until the graphs labeling the leaves of the tree cannot be reduced. See Figure \ref{reductiontree} for an example. 
	
	Fix a reduction tree $\mathcal{R}$ of $G$. Let $L$ be a graph labeling one of the leaves in $\mathcal{R}$. Lemma \ref{contraction} implies that $\mathcal{F}_{\widetilde{L}}$ is integrally equivalent to the standard simplex, so the flow polytopes of the graphs labeling the leaves of $\mathcal{R}$ dissect $\mathcal{F}_{\widetilde{G}}$ into unimodular simplices. Consequently, all dissections we consider in this paper will be dissections into unimodular simplices. By \newword{full-dimensional leaves} of $\mathcal{R}$, we mean the leaves $L$ with $\#E(L)=\#E(G)$. By \newword{lower-dimensional leaves} we mean all other leaves $L$ of $\mathcal{R}$. Note that the full-dimensional leaves correspond to top-dimensional simplices in the dissection of $\mathcal{F}_{\widetilde{G}}$, and the lower-dimensional leaves index intersections of the top-dimensional simplices. The dissections produced by a reduction tree are not generally triangulations, due to how leaves on different sides of the reduction tree can intersect.
	
	Recall the \newword{normalized volume} of a polytope is the usual Euclidean volume scaled by the volume of a unimodular simplex in the affine span of the polytope. Since all simplices $\mathcal{F}_{\widetilde{L}}$ of leaves in a reduction tree are unimodular, we have the following result.
	
	\begin{corollary}
		The normalized volume of $\mathcal{F}_{\widetilde{G}}$ equals the number of full-dimensional leaves in any reduction tree of $G$.
	\end{corollary}
	
	\begin{figure}[ht]
		\centering
		\includegraphics[scale=1]{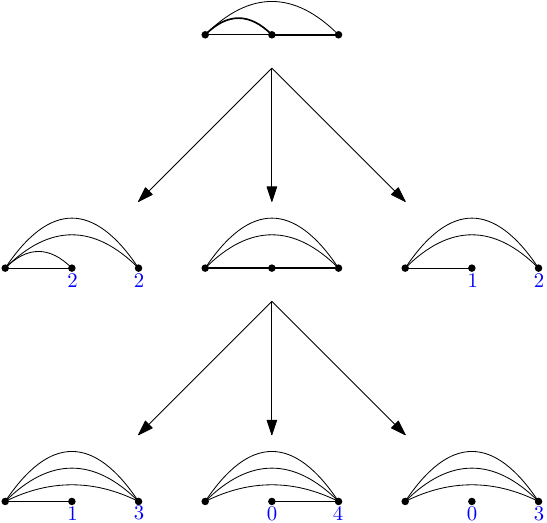}
		\caption{A reduction tree for a graph on three vertices. The edges involved in each reduction are shown in bold. The left-degree sequences of the leaves are displayed below each leaf.}
		\label{reductiontree}
	\end{figure}
	
	\subsection{Left-Degree Sequences}
	
	 Let $G$ be a graph on $[0,n]$, and let $\mathcal{R}$ be a reduction tree of $G$. For each leaf $L$ of $\mathcal{R}$, consider the \newword{left-degree sequence} 
	 \[\left (\indegL(1),\indegL(2),\ldots, \indegL(n) \right ).\] 
	 By \newword{full-dimensional} sequences, we will mean left-degree sequences of full-dimensional leaves of $\mathcal{R}$. Although the actual leaves of a reduction tree are dependent on the individual reductions performed, we prove in Theorem \ref{theoremA} that the left-degree sequences are not.

	\begin{example}
		Any reduction tree of $K_4$ has the full-dimensional left-degree sequences 
		\begin{align*}
			\{(0,0,6),(0,0,6),(0,1,5),(0,1,5),(0,2,4),(0,2,4),(0,3,3),(1,0,5),(1,1,4),(1,2,3) \}
		\end{align*}
	\end{example}

	\section{Triangular arrays and left-degree sequences}
	\label{sec3}
	In this section, we expand the technique described in \cite{prod} that characterized left-degree sequences of full-dimensional leaves in a specific reduction tree of any graph. Given a graph $G$, we construct this reduction tree $\mathcal{T}(G)$. We give a characterization of the left-degree sequences of all leaves of this reduction tree, not just the full-dimensional ones. We then connect this characterization to flow polytopes. The main result of this section is Corollary \ref{flows}, where we provide a partition of the left-degree sequences of $\mathcal{T}(G)$ and biject each block to the set of integer points in a flow polytope.
	\medskip
			
	For simplicity, throughout this section we restrict to the case where $G$ is a simple  graph on the vertex set $[0,n]$. The set $\mathrm{Sol}_{G}(F)$  is defined in Definition \ref{solg} for simple graphs. We address the more technical general case in Section \ref{sec6} and prove Theorem~\ref{theoremA}.

	We begin by generalizing \cite[Lemma 3]{prod} to include the descriptions of the lower dimensional leaves of reductions performed at a special vertex $v$. The proof is a straightforward  generalization of that of  \cite[Lemma 3]{prod}, illustrated in Figure \ref{cornerstonelemmaexample}. 
	The key to the proof is the special reduction order, whereby we always perform a reduction on the longest edges possible that are incident to the vertex at which we are reducing (the length of an edge being the absolute value of the difference of its vertex labels).
	\begin{lemma}
		\label{cornerstonelemma}
		Assume $G$ has a distinguished vertex $v$ with $p$ incoming edges and one outgoing edge $(v,u)$. If  we perform all reductions possible which involve only edges incident to $v$ in the special reduction order, then we obtain graphs $H_i$ for $i\in [p+1]$, and $K_j$ for $j\in[p]$, with 
		\begin{align*}
			(\indegHi(v),\indegHi(u))&=(p+1-i,\,\indegG(u)-1+i),\\
			(\indegKj(v),\indegKj(u))&=(p-j,\,\indegG(u)-1+j).
		\end{align*}
	\end{lemma}
	Note that the previous lemma vacuously yields only $H_1=G$ if $p=0$.
	
	\begin{figure}[h]
		\begin{center}
			\includegraphics[scale=.8]{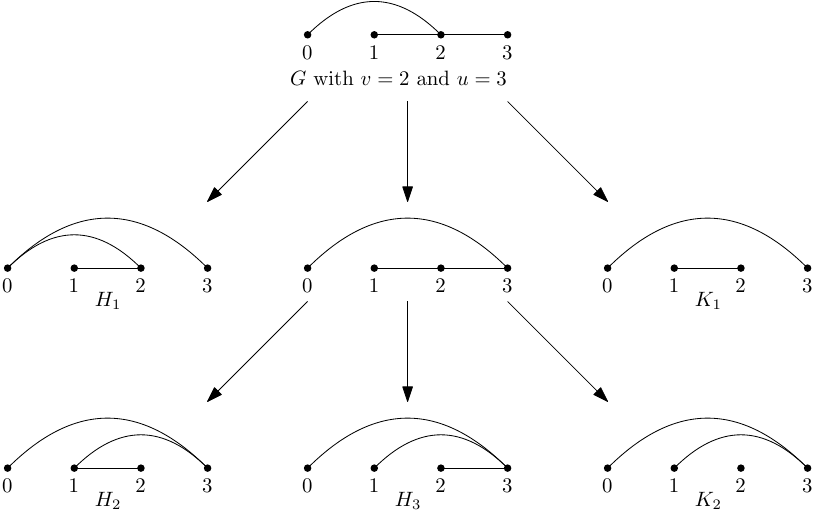}
		\end{center}
		\caption{The graphs $H_i$ and $K_j$ of Lemma \ref{cornerstonelemma}.}
		\label{cornerstonelemmaexample}
	\end{figure}
	
	We now construct a specific reduction tree $\mathcal{T}(G)$ and characterize the left-degree sequences of its leaves. Denote by $I_i$ the set of incoming edges to vertex $i$ in $G$. Let $V_i$ be the set of vertices $k$ with $(k,i)\in I_i$, and let $G[0,i]$ be the restriction of $G$ to the vertices $[0,i]$. For any reduction tree $\mathcal{R}$, by $\inseq(\mathcal{R})$ we mean the multiset of left-degree sequences of the leaves of $\mathcal{R}$. Since we will build $\mathcal{T}(G)$ inductively from $\mathcal{T}(H)$ for smaller graphs $H$, it is convenient to let $\inseq^n(\mathcal{R})$ denote the multiset $\inseq(\mathcal{R})$ with each sequence padded on the right with zeros to have length $n$.
	
	We proceed using the following algorithm, analogous to the one described in \cite{prod}.
	\begin{itemize}
		\item For the base case, define the reduction tree $\mathcal{T}(G[0,1])$ to be the single leaf $G[0,1]$.
		Hence, \[\mathrm{InSeq}(\mathcal{T}(G[0,1]))=\{(\indegG(1)) \}.\]
		
		\item Having built $\mathcal{T}(G[0,i])$, construct the reduction tree $\mathcal{T}(G[0,i+1])$ from $\mathcal{T}(G[0,i])$ by appending the vertex $i+1$ and the edges $I_{i+1}$ to all graphs in $\mathcal{T}(G[0,i])$ and then performing reductions at each vertex in $V_{i+1}$ on all graphs corresponding to the leaves of $\mathcal{T}(G[0,i])$ in the special reduction order as described above Lemma \ref{cornerstonelemma}.
		
		\item Let $V_{i+1}=\{i_1< i_2<\cdots <  i_k\}$ and let $(s_1,\ldots, s_{n})\in\mathrm{InSeq}^{n}(\mathcal{T}(G[0,i]))$. Applying Lemma \ref{cornerstonelemma} to each of the vertices $i_1,\ldots,i_k$, we see that the leaves of $\mathcal{T}(G[0,i+ 1])$ which are descendants of the graph with $n$-left-degree sequence $(s_1,\ldots, s_{n})$ in $\mathcal{T}(G[0,i])$ will have $n$-left-degree sequences exactly given by
		\[  
		(s_1,\ldots, s_{n})+ v^{i+1}[i_1]+\cdots+v^{i+1}[i_k]
		\]
		
		\medskip
		\noindent where $v^{i+1}[i_l]\in S_1(i_l)\cup S_2(i_l)$ and $S_1$, $S_2$ are given by
		\medskip
		\begin{align*}
			S_1(i_l)&=\{(c_1,\ldots, c_{n})\mid  c_j=0 \mbox{ for } j\notin \{i_l,i+1 \}, c_{i_l}=1-s, \mbox{ and } c_{i+1}=s-1 \mbox{ for } s\in [s_{i_l}+1]\}, \\
			S_2(i_l)&=\{(c_1,\ldots, c_{n})\mid  c_j=0 \mbox{ for } j\notin \{i_l,i+1 \}, c_{i_l}=-s, \mbox{ and } c_{i+1}=s-1 \mbox{ for } s\in [s_{i_l}]\}. 
		\end{align*}
	\end{itemize}
	
	\medskip
	\begin{definition}
		\label{specialreductiontree}
		For a simple graph $G$ on $[0,n]$, denote by $\mathcal{T}(G)$ the specific reduction tree constructed using the algorithm described above. Denote by $\mathrm{LD}(G)$ the multiset $\inseq(\mathcal{T}(G))$.
	\end{definition}

	We prove the following surprising property of $\LD(G)$ in Section 6, where we drop the assumption that $G$ be simple.
	\begin{letteredtheorem}
		\label{theoremA}
		Let $G$ be any (not necessarily simple) graph on $[0,n]$. Then for any reduction tree $\mathcal{R}$ of $G$,
		\[\mathrm{LD}(G)=\inseq(\mathcal{R}).\]
	\end{letteredtheorem}
	
	\begin{definition}
		\label{def:triGempty}
		To each leaf $L$ of $\mathcal{T}(G)$, associate the triangular array of numbers $\Arr(L)$ given by 
		\[
		\begin{tabular}{llllll}
		$a_{n1}$ & $a_{n-1,1}$ & $\cdots $&$ a_{31}$&$ a_{21}$&$a_{11}$\\
		$a_{n2}$ &$ a_{n-1,2}$ &$ \cdots $&$ a_{32}$&$a_{22}$& \\ 
		\hspace{2ex}$\vdots$&\hspace{2ex}$\vdots$&$\Ddots$&& & \\
		$a_{n,n-1}$&$a_{n-1, n-1}$&&&&\\
		$a_{nn}$&&&&&
		\end{tabular}
		\]
		
		\noindent where $(a_{i1}, a_{i2}, \ldots, a_{ii})$ is the left-degree sequence of the leaf of $\mathcal{T}(G[0,i])$ preceding (or equaling if $i=n$) $L$ in the construction of $\mathcal{T}(G)$.
	\end{definition}
	
	\begin{theorem}[\cite{prod}, Theorem 4]
		\label{arrayconstraints}
		The arrays $\Arr(L)$ for full-dimensional leaves $L$ of $\mathcal{T}(G)$ are exactly the nonnegative integer solutions in the variables $\{a_{ij}\mid  1\leq j\leq i\leq n \}$ to the constraints
		\begin{itemize}
			\item $a_{11}= \#E(G[0,1])$
			\item $a_{ij}\leq a_{i-1,j} \mbox{ if } (j,i)\in E(G)$
			\item $a_{ij}=a_{i-1,j}$ \mbox{ if } $(j,i)\notin E(G)$ 
			\item $a_{ii}=\#E(G[0,i]) - \sum_{k=1}^{i-1}{a_{ik}}$
		\end{itemize}
	\end{theorem}
	\medskip
	
	\begin{example}
		\label{triangulararrayofconstraintsexample}
		If $G$ is the graph on $[0,4]$ with \[E(G)=\{(0,1),(0,2),(1,2),(2,3),(2,4),(3,4)\},\] 
		then Theorem \ref{arrayconstraints} gives the inequalities
		\begin{align*}
				&0\leq a_{41}=a_{31}= a_{21}\leq a_{11}=1\\
				&0\leq a_{42}\leq a_{32}\leq a_{22}=3-a_{21}\\
				&0\leq a_{43}\leq a_{33}= 4-a_{31}-a_{32}\\
				&0\leq a_{44}=6-a_{41}-a_{42}-a_{42}
		\end{align*}
		The first columns
		\[(a_{41},a_{42},a_{43},a_{44}) \] of solutions to these inequalities are exactly the full-dimensional left-degree sequences of $G$.
	\end{example}
	Given a graph $G$, we write the constraints specified in Theorem \ref{arrayconstraints}  in the form shown in Example \ref{triangulararrayofconstraintsexample} and call them the \newword{triangular constraint array} of $G$. We proceed by generalizing triangular constraint arrays to encode the lower-dimensional leaves of $\mathcal{T}(G)$ as well.
	
	\begin{definition} 
		\label{solg}
		Denote by $\Tri_G(\emptyset)$, or when the context is clear, by $\Tri(\emptyset)$,  the triangular constraint array of $G$. For each subset $F\subseteq E(G\backslash 0)$ (recall that $G$ is simple in this section), define a constraint array $\Tri(F)$ by modifying $\Tri(\emptyset)$ as follows: for each $(j,i)\in F$ and each ordered pair $(m,j)$ with $n\geq m\geq i$, replace each occurrence (anywhere in the inequalities) of $a_{mj}$ by $a_{mj}+1$ and add 1 to the constant at the leftmost edge of row $j$.
		Denote by  $\Sol_G(F)$, or  when the context is clear, by $\Sol(F)$, the collection of all integer solution arrays to the constraints $\Tri(F)$.
	\end{definition}
	
	\begin{example}
		\label{generalizedarrayexample}
		With $G$ as in Example \ref{triangulararrayofconstraintsexample} and $F=\{(2,3),(2,4),(3,4)\}$, we have\\
		\[ \Tri(F):
		\begin{tabular}{l}
		0$\leq a_{41}=a_{31}= a_{21}\leq a_{11}=1$\\
		2$\leq a_{42}+2\leq a_{32}+1\leq a_{22}=3-a_{21}$\\
		1$\leq a_{43}+1\leq a_{33}= 3-a_{31}-a_{32}$\\
		0$\leq a_{44}=3-a_{41}-a_{42}-a_{43}$
		\end{tabular}
		\]
	\end{example}
	
	\noindent The characterization of $\mathrm{LD}(G)=\inseq(\mathcal{T}(G))$ given in the construction of $\mathcal{T}(G)$ implies the following theorem.
	
	\begin{theorem}
		\label{multisetequality}
		The leaves of $\mathcal{T}(G)$ are in bijection with the multiset union of solutions to the arrays $\Tri(F)$, that is 
		
		\[\{\Arr(L)\mid  L \mbox{ is a leaf of }\mathcal{T}(G)\} = \bigcup_{F\subseteq E(G\backslash 0)} \mathrm{Sol}_G(F).\]
		
		\noindent In particular, $\mathrm{LD}(G)$ is the (multiset) image of the right-hand side under the map that takes a triangular array to its first column  $\left (a_{n1}, \ldots, a_{nn}\right)$.
	\end{theorem}
	
	\begin{definition}
		\label{ldsequencesfromF}
		For any $F\subseteq E(G\backslash 0)$, denote by $\LD(G,F)$ the submultiset of $\mathrm{LD}(G)$ consisting of sequences occurring as the first column of an array in $\Sol(F)$. 
	\end{definition}

	As a consequence of Theorem \ref{multisetequality}, \[\mathrm{LD}(G)=\bigcup_{F\subseteq E(G\backslash 0)}\LD(G,F). \]
	
	\begin{remark}
		Combinatorially, we can  think of $\LD(G,F)$  in the following way. Construct the reduction tree $\mathcal{T}(G)$ of $G$. Take any graph $H$ appearing as a node of $\mathcal{T}(G)$. Let $H$ have descendants $H_1$, $H_2$ and $H_3$ in $\mathcal{T}(G)$ obtained by the reduction on edges $(i,j)$ and $(j,k)$ in $H$ with $i<j<k$, so that $H_3$ has edge set $E(H)\backslash \{(i,j), (j,k)\} \cup \{(i,k)\}$. Label the edge in $\mathcal{T}(G)$ between $H$ and $H_3$ by $(j,k)$. To each leaf $L$ of $\mathcal{T}(G)$, associate the set of all labels on the edges of the unique path from $L$ to the root $G$ of $\mathcal{T}(G)$. The left-degree sequences of leaves assigned a set $F$ in this manner are exactly the elements of the multiset $\LD(G,F)$.
	\end{remark}
	 
	\begin{figure}[ht]
		\includegraphics[scale=.9]{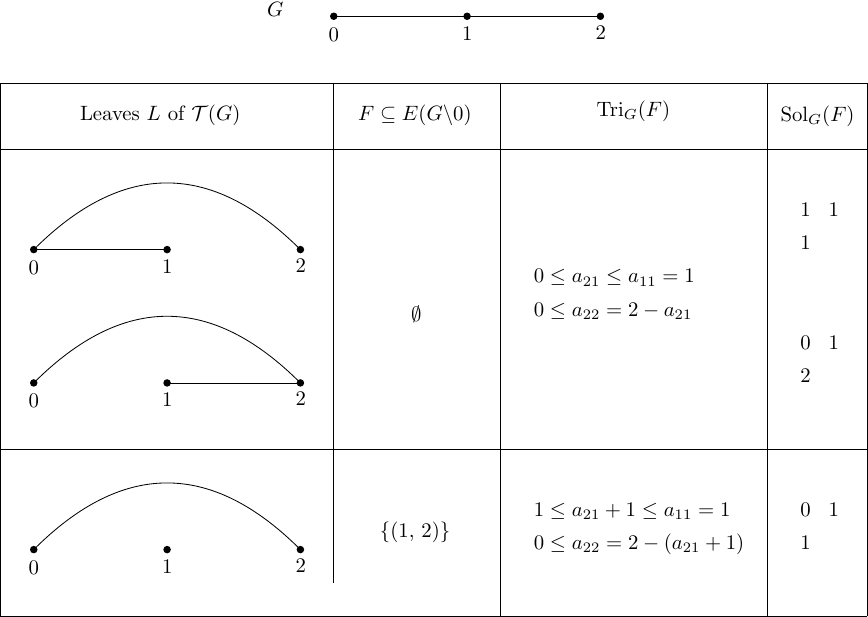}
		\caption{A small example demonstrating Theorem \ref{multisetequality}. In general, $\Sol_G(F)$ will be empty for many $F$.}
	\end{figure}
	
	To understand the multisets $\Sol(F)$ and $\LD(G,F)$, we connect the constraint arrays $\Tri(F)$ to flow polytopes. We begin by investigating the case where $G=K_{n+1}$ is the complete graph on $[0,n]$. Given $F\subseteq E(K_{n+1}\backslash 0)$, consider the numbers
	
	\begin{align}
		\label{fijnumbers} 
		f_{ij}=\#\{(j,k)\in F\mid  k\leq i\}. 
	\end{align}
	
	Observe that for each $F\subseteq E(K_{n+1}\backslash 0)$, $\Tri(F)$ is obtained from $\Tri(\emptyset)$ by replacing $a_{ij}$ in $\Tri(\emptyset)$ by $a_{ij}+f_{ij}$ and replacing the 0 in the leftmost spot of row $j$ by $f_{nj}$. Also note that $f_{jj}=0$ for each $j$. Thus, $\Tri(F)$ is given by
%	Modify $\Tri(F)$ to obtain a new constraint array denoted $A_{K_{n+1}}(F)$ with the same solutions by subtracting $f_{nj}$ from each term in row $j$ for each $j$, so that the leftmost column becomes all zeros.
%	For notational compactness, let $b_{ij}=a_{ij}+f_{ij}$. $A_{K_{n+1}}(F)$ is given by
	\begin{align*}
			&f_{n1}\leq a_{n1}+f_{n1}\leq \cdots \leq  a_{21} +f_{21}\leq a_{11}+f_{11}= \#E(K_{n+1}[0,1])\\
			&f_{n2}\leq a_{n2}+f_{n2}\leq  \cdots \leq a_{22}+f_{22}= \#E(K_{n+1}[0,2])-a_{21}-f_{21} \\ 
			&\phantom{f_{n1}}\phantom{\leq}\vdots \mbox{\hspace{7ex}}\vdots\mbox{\hspace{7ex}}\Ddots  \\
			&f_{nn}\leq a_{nn} +f_{nn}= \#E(K_{n+1})-\displaystyle\sum_{k=1}^{n-1}{a_{nk}}-\displaystyle\sum_{k=1}^{n-1}{f_{nk}}
	\end{align*}
	
	Note that the real solution set in variables $\{a_{ij}\}$ to $\Tri(F)$ is a polytope in $\mathbb{R}^{\binom{n+1}{2}}$. For any constraint array $A$, denote by $\poly(A)$ the \newword{polytope defined by the inequalities in $A$}. We now work toward showing that the polytopes $\mathrm{Poly}(\Tri_G(F))$ are integrally equivalent to flow polytopes. We first continue analyzing the case of the complete graph. Fix $F\subseteq E(K_{n+1}\backslash 0)$.
	
	For $\{(i,j)\mid  1\leq j<i\leq n \}$, we introduce (nonnegative) slack variables $z_{ij}$ to convert the inequalities in $\poly(\Tri(F))$ into equations $Y_{ij}$, given by	
	\[
	Y_{ij}:\begin{cases}
	a_{ij}+f_{ij}+z_{ij}=a_{i-1,j}+f_{i-1,j} & \text{ if } i>j\\
	\displaystyle \sum_{k=1}^{i}a_{ik}+\sum_{k=1}^{i}f_{ik} = \#E(K_{n+1}[0,i])& \text{ if } i=j.
	\end{cases}
	\]
	
	\noindent Define an equivalent system of equations $\{Z'_{ij} \}$ by setting
	\[
	Z'_{ij}:\begin{cases}
	Y_{ij} & \text{ if } i>j \mbox{ or } i=j=1\\
	Y_{ij} - Y_{i-1,j-1} - \displaystyle \sum\limits_{k=1}^{j-1}{Y_{jk}} & \text{ if } i=j>1.
	\end{cases}
	\]
	We then modify each equation $Z'_{ij}$ by rearranging negated terms to get equations $Z_{ij}$ given by 
	\[
	Z_{ij}:\begin{cases}
	a_{ij}+z_{ij}=a_{i-1,j}+f_{i-1,j}-f_{ij} & \text{ if } i>j\\
	a_{ij}=\indegK(1) &\text{ if } i=j=1\\
	a_{ij} = \indegK(j)+ \displaystyle \sum\limits_{k=1}^{j-1}{z_{jk}} & \text{ if } i=j>1
	\end{cases}
	\]
	We now construct a graph $\gr(K_{n+1})$ whose flow polytope will be given by the equations $Z_{ij}$ (plus the conditions  $z_{ij}\geq 0$). Let the vertices of $\gr(K_{n+1})$ be 
	\[\{v_{ij}\mid  1\leq j\leq i\leq n \}\cup \{v_{n+1,n+1}\}\]
	with the ordering $v_{11}<v_{21}<\cdots < v_{n1}<v_{22}<\cdots<v_{nn}<v_{n+1,n+1}$.
	
	\noindent Let the edges of $\gr(K_{n+1})$ be labeled by the flow variables $a_{ij}$ and $z_{ij}$. 
	Set $E(\gr(K_{n+1}))=E_a\cup E_z$ where
	\begin{align*}
	&E_a \mbox{ consists of edges }a_{ij}:v_{ij}\to v_{i+1,j} \mbox{ for } 1\leq j \leq i\leq n \mbox{ and }\\
	&E_z \mbox{ consists of edges }z_{ij}:v_{ij}\to v_{ii} \mbox{ for } 1\leq j < i\leq n 
	\end{align*}
	and we take indices $(n+1,j)$ to refer to $(n+1,n+1)$.\\
	
	\noindent To define the netflow vector $\bm{a}_{K_{n+1}}^F$, we assign netflow $\indegK(j)$ to vertices $v_{jj}$ with $j<n+1$, we assign netflow 
	\[-\#E(K_{n+1})+\sum_{k=1}^{n-1}{f_{nk}}\] 
	to $v_{n+1,n+1}$, and we assign netflow $f_{i-1,j}-f_{ij}$ to each remaining  vertex $v_{ij}$. 
	\\
	
	\noindent The netflow vector $\bm{a}_{K_{n+1}}^F$ is given by reading each row of the triangular array
	\[
	\begin{tabular}{lllll}
	$f_{n-1,1}-f_{n1}$  & $f_{n-2,1}-f_{n-1,1} $&$\hspace{3ex}\cdots $&$f_{11}-f_{21}$&$ \indegK(1)$\\
	$f_{n-1,2}-f_{n2}$  &$ \hspace{5ex} \cdots $&$f_{22}-f_{32} $&$\indegK(2)$& \\ 
	\hspace{7ex}$\vdots$&$\Ddots$&&&  \\
	$\indegK(n)$&&&&\\
	\end{tabular}
	\]
	\noindent right to left starting with the first row, moving top to bottom, and then appending $-\#E(K_{n+1})+\sum_{k=1}^{n-1}{f_{nk}}$ at the end.
	
	\begin{lemma}
		\label{arraysgiveflowpolytopes}
		The polytopes
		\[\mathcal{F}_{\gr(K_{n+1})}(\bm{a}_{K_{n+1}}^F) \mbox{ and } \poly(\Tri(F)) \]
		are integrally equivalent.
	\end{lemma}
	\begin{proof}
		By construction, the flow equation at vertex $v_{ij}$ in $\gr(K_{n+1})$ is exactly the equation $Z_{ij}$ for $(i,j)\neq (n+1,n+1)$. At $v_{n+1,n+1}$, the flow equation is $Y_{nn}$, which follows from the equations $Z_{ij}$ and adds no additional restrictions. The result now follows from the fact that the transformation from $\{Y_{ij}\}_{i,j}$ to $\{Z_{ij}\}_{i,j}$ was unimodular.
	\end{proof}

	\begin{figure}
		\includegraphics[]{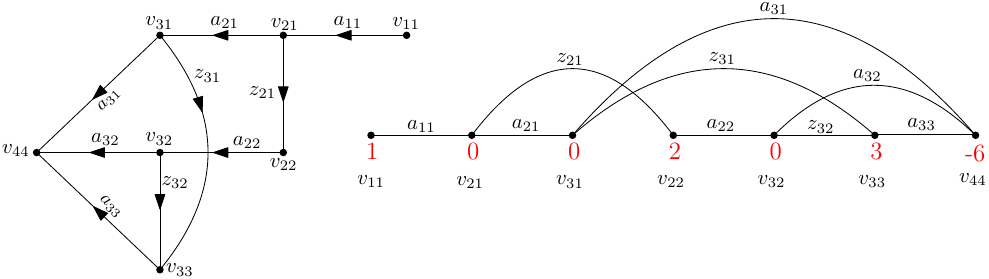}
		\caption{Two drawings of the graph $\gr(K_{n+1})$ of Lemma \ref{arraysgiveflowpolytopes}. The drawing on the right has the netflow vector $\bm{a}_{K_{n+1}}^\emptyset$.}
		\label{completegraph4demonstration}
	\end{figure}
	
	We now generalize Lemma \ref{arraysgiveflowpolytopes} to any simple graph $G$ on $[0,n]$. Note that for $F\subseteq E(G\backslash 0)$, $\Tri_G(F)$ can be obtained from $\Tri_{K_{n+1}}(F)$ by turning certain inequalities into equalities and changing all occurrences of $\#E(K_{n+1}[0,j])$ to $\#E(G[0,j])$ for each $j$. In terms of $\{Z_{ij}\}_{i,j}$, this amounts to setting $z_{ij}=0$ whenever $(j,i)\notin E(G)$. Relative to the graph $\gr(K_{n+1})$, this is equivalent to deleting the edges labeled $z_{ij}$ for $(j,i)\notin E(G)$. Thus, we have the following extension of $\gr(K_{n+1})$.
	
	\begin{definition}
		\label{tagsimplefinalversion}
		For a simple graph $G$ on $[0,n]$ define a graph $\gr(G)$ on vertices 
		\[\{v_{ij}\mid  1\leq j\leq i\leq n \}\cup \{v_{n+1,n+1}\}\]
		ordered $v_{11}<v_{21}<\cdots < v_{n1}<v_{22}<\cdots<v_{nn}<v_{n+1,n+1}$ and with edges 
		$E_a\cup E_z$ where
		\begin{align*}
		&E_a \mbox{ consists of edges }a_{ij}:v_{ij}\to v_{i+1,j} \mbox{ for } 1\leq j \leq i\leq n \mbox{ and }\\
		&E_z \mbox{ consists of edges }z_{ij}:v_{ij}\to v_{ii} \mbox{ for } (j,i)\in E(G\backslash 0).
		\end{align*}
		For any $F\subseteq E(G\backslash 0)$, we define a netflow vector $\bm{a}_G^F$ for $\gr(G)$ by reading each row of the triangular array
		\[
		\begin{tabular}{lllll}
		$f_{n-1,1}-f_{n1}$  & $f_{n-2,1}-f_{n-1,1} $&$\hspace{3ex}\cdots $&$f_{11}-f_{21}$&$ \indegG(1)$\\
		$f_{n-1,2}-f_{n2}$  &$ \hspace{5ex} \cdots $&$f_{22}-f_{32} $&$\indegG(2)$& \\ 
		\hspace{7ex}$\vdots$&$\Ddots$&&&  \\
		$\indegG(n)$&&&&\\
		\end{tabular}
		\]
		\noindent right to left starting with the first row, moving top to bottom, and then appending \\$-\#E(G)+\sum_{k=1}^{n-1}{f_{nk}}$ at the end, where again, $f_{ij}=\#\{(j,k)\in F\mid  k\leq i\}$.
	\end{definition}
	
	We now have the following extension of Lemma \ref{arraysgiveflowpolytopes} to all simple graphs.
	\begin{proposition}
		\label{tapisflowpolytope}
		Let $G$ be a simple graph on $[0,n]$ and $F\subseteq E(G\backslash 0)$. Then,
		$\poly(\Tri_G(F))$ is integrally equivalent to $\mathcal{F}_{\gr(G)}(\bm{a}_G^F)$. In particular, the multiset of solutions $\Sol_G(F)$ to $\Tri_G(F)$ consists precisely of the projections of integral flows on $\gr(G)$ with netflow $\bm{a}_G^F$ onto the edges labeled $\{a_{ij}\}$.
	\end{proposition}
%	\begin{proof}
%		The transformation from the equations $Y_{ij}$ to $Z_{ij}$ is unimodular, so integer points are preserved.
%	\end{proof}
	
	\begin{example}
		Let $G$ be the graph on $[0,4]$ with
		\[E(G)=\{(0,1),(0,2),(1,2), (2,3),(2,4),(3,4) \}\] and $F=\{(2,3)\}$. The graph $\gr(G)$ and its netflow vector $\bm{a}_G^F$ are shown in Figure \ref{fig:grg}.
		
		Observe that contracting the edges $\{a_{11}, a_{21}, a_{31}, a_{22}, a_{32},a_{33} \}$ in $\gr(G)$ yields the graph shown in Figure \ref{fig:grgcontracted}, which is exactly $\widetilde{G}\backslash \{s,0\}$. The next result shows that this occurs in general.
	\end{example}
	\begin{figure}[ht]
		\includegraphics[scale=.9]{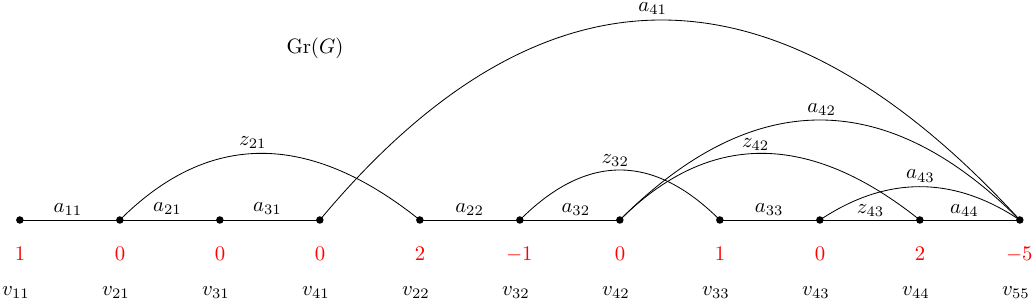}
		\caption{The graph $\gr(G)$ when $E(G)=\{(0,1),(0,2),(1,2), (2,3),(2,4),(3,4) \}$.}
		\label{fig:grg}
	\end{figure}

	\begin{figure}[ht]
		\includegraphics[scale=1]{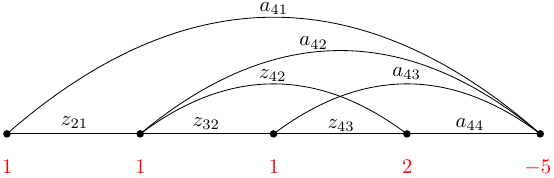}
		\caption{The graph $\gr(G)/\{a_{11}, a_{21}, a_{31}, a_{22}, a_{32},a_{33} \}$}
		\label{fig:grgcontracted}
	\end{figure}

	\medskip
	 For a graph $G$ and a subset $F\subseteq E(G\backslash 0)$, view $F$ as a subgraph of $G$ on the same vertex set. Note that for each $j$,
	 \[f_{nj}=\#\{(j,k)\in F\mid  k\leq n\}=\outdegF(j)\]
	 and the number 
	 \[-\#E(G)+\sum_{k=1}^{n-1}{f_{nk}}\]
	 appearing as the last entry of $\bm{a}_G^F$ equals $-\#E(G\backslash F)$.
	
	\begin{theorem}
		\label{isomorphism}
		Let $G$ be a simple graph on $[0,n]$ and $F\subseteq E(G\backslash 0)$. Then, the flow polytopes 
		\[\mathcal{F}_{\gr(G)}\left (\bm{a}_G^F \right ) \mbox{ and } \mathcal{F}_{\widetilde{G}\backslash\{s,0\}}\left (\mathrm{indeg}_G(1)-\mathrm{outdeg}_F(1),\ldots, \mathrm{indeg}_G(n)-\mathrm{outdeg}_F(n), -\#E(G\backslash F) \right )\] are integrally equivalent.
	\end{theorem}
	
	\begin{proof}
		First, note that in $\gr(G)$, the edges $\{a_{ij}\mid  i<n\}$ are each the only edges incoming to their target vertex. Contracting these edges via Lemma \ref{contraction}  identifies vertices $v_{ij}$ and $v_{i'j}$. Label the representative vertices $v_{jj}$ by $j$ for $j\in [n]$ and $v_{n+1,n+1}$ by $t$. The remaining edges are 
		\[z_{ij}:j\to i \mbox{ for } (j, i)\in E(G) \mbox{ and } a_{nj}: j\to t \mbox{ for } j\in [n],\]
		
		\noindent which are exactly the edges of $\widetilde{G}- \{s,0\}$. 
		\\
		
		\noindent Viewing the netflow vector $\bm{a}_G^F$ as the array 
		\[
		\begin{tabular}{lllll}
		$f_{n-1,1}-f_{n1}$  & $f_{n-2,1}-f_{n-1,1} $&$\hspace{3ex}\cdots $&$f_{11}-f_{21}$&$ \indegG(1)$\\
		$f_{n-1,2}-f_{n2}$  &$ \hspace{5ex} \cdots $&$f_{22}-f_{32} $&$\indegG(2)$& \\ 
		\hspace{7ex}$\vdots$&$\Ddots$&&&  \\
		$\indegG(n)$&&&&\\
		$-\#E(G\backslash F),$&&&&
		\end{tabular}
		\]
		Lemma \ref{contraction} implies the entries of the netflow vector after contracting are given by reading the sums of each row from top to bottom.
		 
	\end{proof}

	Recall from Definition \ref{ldsequencesfromF} that $\LD(G,F)$ is the multiset of left-degree sequences in $\inseq(\mathcal{T}(G))$ occurring as the first column $\left (a_{n1}, \ldots, a_{n n}\right)$ of an array in $\Sol(F)$. We now arrive at the main result of this section.
	\begin{corollary}
		\label{flows}
		Let $G$ be a simple graph on $[0,n]$ and $F\subseteq E(G\backslash 0)$. If $\bm{b}_G^F$ is the vector \[\bm{b}_G^F=\left (\indegG(1)-\outdegF(1), \ldots , \, \indegG(n)-\outdegF(n), -\#E(G\backslash F)   \right )\]
		and $\psi$ is the map that takes a flow on $\widetilde{G}\backslash \{s,0\}$ to the tuple of its values on the edges $\{(j,t)\mid j\in[n] \}$, then
		$\LD(G,F)$ equals the (multiset) image under $\psi$ of all integral flows on $\widetilde{G}\backslash \{s,0\}$ with netflow vector $\bm{b}_G^F$.
		
		\noindent In particular, $\LD(G,F)$ is in bijection with integral flows on $\widetilde{G}\backslash \{s, 0\}$ with netflow $\bm{b}_G^F$.
		\end{corollary}
	
	We note that the  preceding result   implies a formula for the Ehrhart polynomial of  flow polytopes of graphs with special source and sink vertices. In particular, a special case of  Theorem \ref{pstheorem} follows readily.
	 
	\begin{theorem}
		\label{pstheoremaltproof}
		Let $G$ be a simple graph on $[0,n]$ and let $d_i=\indeg_G(i)$. Then, the normalized volume of the flow polytope of $\widetilde{G}$ is  
		\begin{align} 
			\label{eq:vol} 
			\mathrm{Vol} \,\,\mathcal{F}_{\widetilde{G}} = K_{\widetilde{G}\backslash \{s,0\}} \left (d_1,\ldots, d_n, \, -\# E(G) \right ). 
		\end{align}
		
		Moreover, the Ehrhart polynomial of $\mathcal{F}_{\widetilde{G}}$ is \
		\begin{align} 
			\label{eq:ehr} 
			\operatorname {Ehr}(\mathcal{F}_{\widetilde{G}},t) = (-1)^{d}\sum_{i=0}^{d} (-1)^i  \left (\mbox{\hspace{1.5ex}}\sum_{\mathclap{\substack{F\subseteq E(G\backslash 0)\\ \#F=d-i}}}  K_{\widetilde{G}\backslash \{s,0\}} \left (\bm{b}_G^F \right ) \right )  {t+i \choose i}, 
		\end{align} where $\bm{b}_G^F= (\indegG(1)-\outdegF(1), \ldots , \, \indegG(n)-\outdegF(n), -\#E(G\backslash F))$ and $d=\#E(\widetilde{G})-\#V(\widetilde{G})+1$ is the dimension of $\mathcal{F}_{\widetilde{G}}$.
	\end{theorem}
	
	\begin{proof}
		From the dissection of $\mathcal{F}_{\widetilde{G}}$ obtained  via the reduction tree $\mathcal{T}(G)$, it follows that \\ $\mathrm{Vol} \,\,\mathcal{F}_{\widetilde{G}} $ is the number of full-dimensional left-degree sequences. By Corollary \ref{flows}, these are in bijection with the integer points in the flow polytope $\mathcal{F}_{\widetilde{G}\backslash \{s,0\}}\left (d_1,\ldots, d_{n}, -\#E(G) \right )$, proving \eqref{eq:vol}.
		
		To prove \eqref{eq:ehr} note that    $\mathcal{F}_{\widetilde{G}}^\circ= \bigsqcup _{\sigma^\circ \in D_{\mathcal{T}(G)}} \sigma^\circ$, where $D_{\mathcal{T}(G)}$ is the set of open simplices corresponding to the leaves of the reduction tree $\mathcal{T}(G)$.  Then,   
		\[\operatorname {Ehr}(\mathcal{F}_{\widetilde{G}}^\circ, t)=\sum _{\sigma^\circ \in D_{\mathcal{T}(G)}} \operatorname {Ehr}({\sigma^\circ}, t).\]
		Since all simplices in $D_{\mathcal{T}(G)}$ are unimodular, it follows that for a $k$-dimensional simplex $\sigma^\circ \in D_{\mathcal{T}(G)}$, 
		\[\operatorname {Ehr}({\sigma^\circ}, t)=\operatorname {Ehr}({\Delta^\circ}, t),\]
		where $\Delta$ is the standard $k$-simplex. By \cite[Theorem 2.2]{br}, $\operatorname {Ehr}({\Delta^\circ}, t)= {t-1 \choose k}.$
		Thus,  
		\[\operatorname {Ehr}(\mathcal{F}_{\widetilde{G}}^\circ, t)=\sum_{i=0}^{\infty} f_i  {t-1 \choose i},\] where $f_i$ is the number of $i$-simplices in $D_{\mathcal{T}(G)}$. For $i\in [0,d]$, 
		\[f_i=\sum_{\mathclap{\substack{F\subseteq E(G\backslash 0)\\ \#F=d-i}}} \#LD(G, F).\]  
		Corollary \ref{flows} then implies   
		\[f_i= \sum_{\mathclap{\substack{F\subseteq E(G\backslash 0)\\ \#F=d-i}}}  K_{\widetilde{G}\backslash \{s,0\}} \left (\bm{b}_G^F \right ) \text{ for } i \in [0,d].\]
		Therefore, 
		\[ \operatorname {Ehr}(\mathcal{F}_{\widetilde{G}}^\circ, t)=\sum_{i=0}^{d} \left (\mbox{\hspace{1.5ex}}\sum_{\mathclap{\substack{F\subseteq E(G\backslash 0)\\ \#F=d-i}}}  K_{\widetilde{G}\backslash \{s,0\}} \left (\bm{b}_G^F \right )\right )  {t-1 \choose i}.\]

		\noindent From the Ehrhart-Macdonald reciprocity \cite[Theorem 4.1]{br} \[\operatorname {Ehr}(\mathcal{F}_{\widetilde{G}},t)=(-1)^{d} \operatorname {Ehr}(\mathcal{F}_{\widetilde{G}}^{\circ},-t),\]
		it follows that 
		\begin{align*}
			\operatorname {Ehr}(\mathcal{F}_{\widetilde{G}},t)&=(-1)^{d}\sum_{i=0}^{d} \left (\mbox{\hspace{1.5ex}}\sum_{\mathclap{\substack{F\subseteq E(G\backslash 0)\\ \#F=d-i}}}  K_{\widetilde{G}\backslash \{s,0\}} \left (\bm{b}_G^F \right )\right )  {-t-1 \choose i}\\
			&=(-1)^{d}\sum_{i=0}^{d} (-1)^i  \left (\mbox{\hspace{1.5ex}}\sum_{\mathclap{\substack{F\subseteq E(G\backslash 0)\\ \#F=d-i}}}  K_{\widetilde{G}\backslash \{s,0\}} \left (\bm{b}_G^F \right ) \right )  {t+i \choose i}.
		\end{align*}
	\end{proof}

\section{Newton polytopes of left-degree polynomials}
\label{sec4}

In this section, we study the Newton polytopes of polynomials $L_G(\bm{t})$ built from left-degree sequences (see Definition \ref{ldpolynomialdefinition}). We first show that each of these polynomials have SNP (Definition \ref{snp}). Then, we investigate the Newton polytopes of their homogeneous components and certain homogeneous subcomponents. We prove that these Newton polytopes are generalized permutahedra. Our main results can be summarized as:

\begin{letteredtheorem}
	\label{theoremB}
	Let $G$ be a graph on $[0,n]$. Then the left-degree polynomial $L_G(\bm{t})$ has SNP, and the Newton polytope of each homogeneous component $L_G^k(\bm{t})$ of $L_G(\bm{t})$ of degree $\#E(G)-k$  is a generalized permutahedron.
\end{letteredtheorem}
	Theorems \ref{newtonofldpolynomial}, \ref{cap} and   \ref{ldhomogeneouspiecesnewton} imply Theorem \ref{theoremB}, and contain more detail regarding the elements of Theorem \ref{theoremB}. Recall that for a polynomial $f=\displaystyle \sum_{\alpha\in \mathbb{Z}_{\geq 0}^n}{c_{\alpha}}\bm{t}^{\alpha}$, the \newword{Newton polytope} is 
	\[\mathrm{Newton}(f)=\conv\left( \left \{ \alpha\mid  c_{\alpha}\neq 0 \right \} \right).\] 	

\begin{definition} \label{snp}
	We say a polynomial $f$ has \newword{saturated Newton polytope} (SNP) if $c_{\alpha}\neq 0$ whenever $\alpha\in \mathrm{Newton}(f)$; that is, if the integer points of $\mathrm{Newton}(f)$ are exactly the exponents of monomials appearing in $f$ with nonzero coefficients.
\end{definition}

The question of when a polynomial has SNP is a very natural one, and has recently been investigated for various polynomials from algebra and combinatorics by Monical, Tokcan and Yong in \cite{MTY}.

Recall from Definition \ref{ldsequencesfromF} that for a simple graph $G$ and a subset  $F\subseteq E(G\backslash 0)$, $\LD(G,F)$ denotes the submultiset of $\LD(G)$ consisting of sequences occurring as the first column of an array in $\Sol(F)$. Just as in Section \ref{sec3}, for the remainder of this section we add the simplifying assumption that $G$ has no multiple edges. All of the results of this section are also valid for graphs with multiple edges, with similar proof and notation modifications to those described in Section \ref{sec6}.

\begin{definition}
	\label{ldpolynomialdefinition}
	Let $G$ be a graph on $[0,n]$. For $\alpha\in\LD(G)$, let $\mathrm{codim}(\alpha)=\#E(G)-\sum_{i=1}^{n}{\alpha_i}$. Define the \newword{left-degree polynomial} $L_G(\bm{t})$   in variables $\bm{t}=(t_1,t_2,\ldots, t_n )$ by
	\begin{align*}
		L_G(\bm{t})=\sum_{\alpha \in \LD(G)}(-1)^{\mathrm{codim}(\alpha)}\bm{t}^{\alpha}.
	\end{align*}
	Similarly, for $F\subseteq E(G\backslash 0)$, define $L_{G,F}(\bm{t})$ by
	\begin{align*}
	L_{G,F}(\bm{t})=\sum_{\alpha \in \LD(G,F)}(-1)^{\mathrm{codim}(\alpha)}\bm{t}^{\alpha}=(-1)^{\#F}\sum_{\alpha \in \LD(G,F)}\bm{t}^{\alpha}.
	\end{align*}
\end{definition}
	Note that the $(-1)^{\mathrm{codim}(\alpha)}$ in Definition \ref{ldpolynomialdefinition} has no effect on the Newton polytope. It is inherited from the definition of right-degree polynomials utilized in \cite{pipe1}, which was designed to agree with Grothendieck polynomials.
	
	Restating Theorem \ref{multisetequality}  in terms of left-degree sequences gives the multiset union decomposition
\[\LD(G)=\bigcup_{F\subseteq E(G\backslash 0)}\LD(G,F). \]
Relative to Newton polytopes, this implies
\begin{align}
	\mathrm{Newton}(L_G(\bm{t})) = \conv\left (\bigcup_{F\subseteq E(G\backslash 0)}\mathrm{Newton}\left (L_{G,F}(\bm{t})\right ) \right ).
\end{align}

We first study the polytope $\mathrm{Newton}(L_G(\bm{t}))$, and then the component pieces $\mathrm{Newton}\left (L_{G,F}(\bm{t})\right )$. To start, we define a new constraint array.

\begin{definition}
	Let $G$ be a simple graph on $[0,n]$. Proceed as follows:
	\begin{itemize}
		\item Start with the triangular constraint array $\Tri_G(\emptyset)$ of $G$  as in Theorem \ref{arrayconstraints}.
		
		\item Replace the zero on the left of row $j$ by $y_{nj}+y_{n-1,j}+\cdots + y_{j+1,j}$ for $j\in [n-1]$, so the zero on the left in row $n$ is left unchanged. 
		
		\item For each $(i,j)$ with $n\geq i>j\geq 1$, replace all occurrences of $a_{ij}$ in the array by $a_{ij}+\sum_{k=j+1}^i y_{kj}$.
		
		\item For every $(j,i)\notin E(G\backslash 0)$, set $y_{ij}=0$ throughout.
	\end{itemize}
	We refer to this array as the \newword{augmented constraint array of} $G$ and view it as having variables $a_{ij}$ and $y_{ij}$ subject to the additional constraints that for all $1\leq j<i\leq n$,
	\[0\leq y_{ij}\leq 1.\]
\end{definition}

\begin{example}
	If $G$ is the graph on vertex set $[0,4]$ with \newline $E(G)=\{(0,1),(0,2),(1,2),(2,3),(2,4),(3,4) \}$, then we start with the constraints
	\begin{align*}
	0&\leq a_{41}=a_{31}= a_{21}\leq a_{11}=1\\
	0&\leq a_{42}\leq a_{32}\leq a_{22}=3-a_{21}\\
	0&\leq a_{43}\leq a_{33}= 4-a_{31}-a_{32}\\
	0&\leq a_{44}=6-a_{41}-a_{42}-a_{43}
	\end{align*}
	After performing the modifications, we arrive at
	\begin{align*}
	y_{21}&\leq a_{41}+y_{21}=a_{31}+y_{21}= a_{21}+y_{21}\leq a_{11}=1\\
	y_{42}+y_{32}&\leq a_{42}+y_{42}+y_{32}\leq a_{32}+y_{32}\leq a_{22}=3-a_{21}-y_{21}\\
	y_{43}&\leq a_{43}+y_{43}\leq a_{33}= 4-a_{31}-y_{21}-a_{32}-y_{32}\\
	0&\leq a_{44}=6-a_{41}-y_{21}-a_{42}-y_{42}-y_{32}-a_{43}-y_{43}
	\end{align*}
\end{example}

	Analogous to Lemma \ref{arraysgiveflowpolytopes}, we now work toward showing that $\poly(A)$ is integrally equivalent to a flow polytope. We use the technique with which we constructed $\gr(G)$ in Lemma \ref{arraysgiveflowpolytopes} togetherwith the proof idea of Theorem \ref{isomorphism}. Begin with the case where $G$ is a complete graph. By introducing slack variables $z_{ij}$ for the inequalities in the augmented constraint array (not $0\leq y_{ij}\leq 1$), we get equations $Y_{ij}$ given by
	\begin{align*}
	Y_{ij}:\begin{cases}
	a_{ij}+y_{ij}+z_{ij}=a_{i-1,j} &\mbox{ if } i>j\\
	a_{ij}=\#E(G[0,1]) &\mbox{ if } i=j=1\\
	\displaystyle\sum_{k=1}^{i}{a_{ik}}+ \sum_{m=2}^{i}{\sum_{k=1}^{m-1}{y_{mk}}} = \#E(G[0,i]) &\mbox{ if } i=j>1
	\end{cases}
	\end{align*}
	
	Applying the exact same transformation used in the proof of Lemma \ref{arraysgiveflowpolytopes}, we get equivalent equations $Z_{ij}$ given by
	\begin{align*}
	Z_{ij}:
	\begin{cases}
	a_{ij}+y_{ij}+z_{ij}=a_{i-1,j} &\mbox{ if } i>j\\
	a_{ij}=\indegG(1) &\mbox{ if } i=j=1\\
	a_{ij} = \indegG(i)+ \displaystyle \sum\limits_{k=1}^{i-1}{z_{ik}} &\mbox{ if } i=j>1
	\end{cases}
	\end{align*}
	
	To move from the complete graph to   any simple graph, just set $y_{ij}=0$ and $z_{ij}=0$ whenever $(j,i)\notin E(G)$. We can realize the solutions to the $Z_{ij}$ as points in a flow polytope of some graph. However, to account for the additional restrictions $0\leq y_{ij}\leq 1$, we view it as a \newword{capacitated flow polytope}. This is for convenience and is not mathematically significant since any capacitated flow polytope is integrally equivalent to an uncapacitated flow polytope \cite[Lemma 1]{cap}.
	
	\begin{definition}
		\label{augmentedarraygraph}
		Define the \newword{augmented constraint graph} $\agr(G)$ to have vertex set $\{v_{ij}\mid  1\leq j\leq i\leq n \}\cup \{v_{n+1,n+1}\}$ 
		with  the ordering $v_{11}<v_{21}<\cdots < v_{n1}<v_{22}<\cdots<v_{nn}<v_{n+1,n+1}$ and edge set  
		$E_a\cup E_z \cup E_y$
		\noindent labeled by the variables $a_{ij}$, $z_{ij}$, and $y_{ij}$ respectively, where
		\begin{align*}
		E_a&\mbox{ consists of edges }a_{ij}:v_{ij}\to v_{i+1,j}\mbox{ for } 1\leq j \leq i \leq n,  \\
		E_z&\mbox{ consists of edges }z_{ij}:v_{ij}\to v_{ii}\mbox{ for }(j,i)\in E(G\backslash 0) ,\\
		E_y&\mbox{ consists of edges }y_{ij}:v_{ij} \to v_{n+1,n+1}\mbox{ for } (j,i)\in E(G\backslash 0),
		\end{align*}
		and we take indices $(n+1,j)$ to refer to $(n+1,n+1)$. 
		Define a netflow vector $\bm{a}_G^{\mathrm{aug}}$ by reading each row of the array
		\[
		\begin{tabular}{lllll}
		$0$  & $0 $&$\hspace{3ex}\cdots $&$0$&$ \indegG(1)$\\
		$0$  &$ 0\hspace{5ex} \cdots $&$0 $&$\indegG(2)$& \\ 
		\hspace{7ex}$\vdots$&$\Ddots$&&&  \\
		$\indegG(n)$&&&&\\
		$-\#E(G)$&&&&
		\end{tabular}
		\]
		from right to left and reading the rows from top to bottom.
	\end{definition}

	Denote by $\mathcal{F}_{\agr(G)}^c\left (\bm{a}_G^{\mathrm{aug}}\right )$ the capacitated flow polytope of the graph $\agr(G)$ with netflow $\bm{a}_G^{\mathrm{aug}}$ and with the capacity constraints $0\leq y_{ij}\leq 1$ for all $1\leq j<i\leq n$. By construction, the points in $\mathcal{F}_{\agr(G)}^c\left (\bm{a}_G^{\mathrm{aug}}\right )$ are exactly the solutions to the augmented constraint array of $G$. 
	\begin{definition}
		\label{moddingofG}
		Similar to Theorem \ref{isomorphism}, contracting the edges $\{a_{ij}\mid  1\leq j\leq i<n \}$ of $\agr(G)$ and relabeling the representative vertices $v_{nj}$ by $j$ and $v_{n+1,n+1}$ by $t$, we obtain a graph called the \newword{augmented graph of $G$}. This graph is denoted $\Gaug$ and is defined on vertices $[n]\cup \{t\}$ with labeled  edges $E_a\cup E_z\cup E_y$ where
		\begin{align*}
		E_a&\mbox{ consists of edges }a_{nj}:j\to t\mbox{ for } j\in [n];  \\
		E_z&\mbox{ consists of edges }z_{ij}:j\to i\mbox{ for }(j,i)\in E(G\backslash 0);\\
		E_y&\mbox{ consists of edges }y_{ij}:j \to t\mbox{ for } (j, i)\in E(G\backslash 0).
		\end{align*}
	 
	\end{definition}
	
	\begin{example}
		For $G=K_4$, the graphs $\agr(G)$ and $\Gaug$ are shown below.
	\end{example}
	\begin{center}
		\includegraphics[scale=.8]{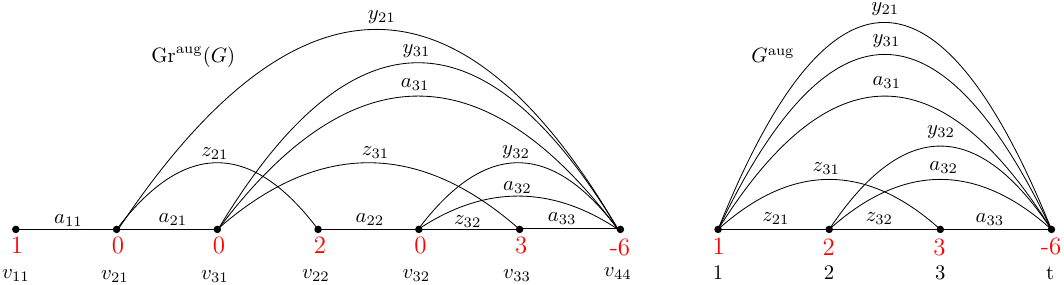}
	\end{center}

\bigskip
Before proceeding, recall the netflow vector 
\[
	\bm{b}_G^F=\left (\indegG(1)-\outdegF(1),\ldots, \indegG(n)-\outdegF(n),-\#E(G\backslash F) \right )
\]
	for any $F\subseteq E(G\backslash 0).$ Denote by $ \mathcal{F}^c_{\Gaug}\left (\bm{b}_G^\emptyset\right )$ the capacitated flow polytope of the graph ${\Gaug}$ with netflow $\bm{b}_G^\emptyset$ and the capacity constraints $0\leq y_{ij}\leq 1$ for all $1\leq j<i\leq n$. 
	
	\begin{theorem}
		\label{newtonofldpolynomial}
		Let $A$ denote the augmented constraint array of $G$ and $\poly(A)$ the polytope defined by the real valued solutions to $A$ with the additional constraints $0\leq y_{ij}\leq 1$ for all $i$ and $j$ with $1\leq j<i\leq n$. Then, the capacitated flow polytopes
		\[\poly(A),\quad \mathcal{F}_{\agr(G)}^c\left (\bm{a}_G^{\mathrm{aug}}\right ),\mbox{ and}\quad \mathcal{F}^c_{\Gaug}\left (\bm{b}_G^\emptyset\right ) \]
		are all integrally equivalent.
	\end{theorem}	
	
	\begin{proof}
		Follows immediately from the constructions of Definitions \ref{augmentedarraygraph} and \ref{moddingofG}.
%		 
%		Thus, it suffices to prove
%		\begin{align*}
%			\mathrm{Newton}(L_G(\bm{t}))=\psi \left ( \mathcal{F}^c_{\Gaug}\left (\bm{b}_G^\emptyset\right ) \right ).
%		\end{align*}
%		where $\psi$ is the projection that takes a flow on $\mathcal{F}^c_{\Gaug}\left (\bm{b}_G^\emptyset\right )$ to its values on the edges labeled $\{a_{nj}\mid j\in [n]\}$. This is accomplished in Theorem \ref{cap}. 
	\end{proof}

			%	If $\rho$ is the projection that maps a solution of A to its values $(a_{n1},\ldots,a_{nn})$, then
	%	\begin{align*}
	%	\mathrm{Newton}(L_G(\bm{t}))=\rho \left ( \poly(A) \right ).
	%	\end{align*}
	%	
		
		\begin{theorem} \label{cap} For $G$ a graph on $[0,n]$, the Newton polytope of the left-degree polynomial $L_G(\bm{t})$ and the capacitated flow polytope $ \mathcal{F}^c_{\Gaug}\left (\bm{b}_G^\emptyset\right )$ satisfy
		\begin{align*}
			\mathrm{Newton}(L_G(\bm{t}))=\psi \left ( \mathcal{F}^c_{\Gaug}\left (\bm{b}_G^\emptyset\right ) \right ),
		\end{align*} where 	where $\psi$ is the projection that takes a flow on $\mathcal{F}^c_{\Gaug}\left (\bm{b}_G^\emptyset\right )$ to its values on the edges labeled $\{a_{nj}\mid j\in [n]\}$.
	\end{theorem}
		
		\begin{proof}
		Let $\alpha\in \LD(G,F)$ for $F\subseteq E(G\backslash 0)$. Consider the set of integer flows on $\Gaug$ such that each edge $y_{ij}$ has flow 1 if $(j,i)\in F$ and zero otherwise. By the construction of $\Gaug$, these are in bijection with the integer flows on $\widetilde{G}\backslash \{s,0\}$ with netflow vector $\bm{b}_G^F$, which in turn are in bijection to $\LD(G,F)$ (Corollary \ref{flows}). Thus $\alpha$ is the projection of a capacitated flow on $\Gaug$ with netflow $\bm{b}_G^\emptyset$.
		
		Conversely, let $\alpha=(\alpha_1,\ldots, \alpha_n)\in \psi \left ( \mathcal{F}^c_{\Gaug}\left (\bm{b}_G^\emptyset\right ) \right ) $ be an integer point. Then, there exists some flow $f$ (not necessarily integral) on $\Gaug$ with netflow $\bm{b}_G^\emptyset$ having the integer values $\alpha_j$ on the $a$-edges $(j,t)$. If we remove these edges and modify the netflow vector accordingly, the new flow polytope we get is the (integrally capacitated) flow polytope of a graph with an integral netflow vector. Any such polytope has integral vertices \cite[Theorem 13.1]{schrijver}. Thus, we can choose $f$ to be an integral flow. 
	
		Since the edges labeled $y_{ij}$ are constrained between 0 and 1, $f$ takes value 0 or 1 on these edges. If we let $F=\{(j,i)\in E(G\backslash 0)\mid f \mbox{ takes value 1 on the edge labeled by } y_{ij} \}$, then $f$ induces a flow on $\widetilde{G}\backslash\{s,0\}$ with netflow vector $\bm{b}_G^F$, so $\alpha\in \LD(G,F)$.
	\end{proof}

	\begin{corollary}
		\label{cor:LGSNP}
		For any graph $G$ on $[0,n]$, $L_G(\bm{t})$ has SNP.
	\end{corollary}
	\begin{proof}
		The second half of the proof of Theorem \ref{cap} demonstrated that any integer point $\alpha\in \psi \left ( \mathcal{F}^c_{\Gaug}\left (\bm{b}_G^\emptyset\right ) \right )$ satisfied $\alpha\in \LD(G,F)$ for some $F$. Thus $\alpha\in \LD(G)$.
	\end{proof}

We now analyze the component polytopes $\mathrm{Newton}(L_{G,F}(\bm{t}))$ and show that they are generalized permutahedra. We first briefly recall the relevant definitions from \cite{genperms}.

A \newword{generalized permutahedron} is a deformation of the usual permutahedron obtained by parallel translation of the facets. Generalized permutahedra are parameterized by real numbers $\{z_I\}_{I\subseteq [n]}$ with $z_\emptyset=0$ and satisfying the supermodularity condition \[z_{I\cup J}+z_{I\cap J}\geq z_{I}+z_{J} \mbox{ for any } I,J\subseteq [n].\] 
For a choice of parameters $\{z_I\}_{I\subseteq [n]}$, the associated generalized permutahedron $P_n^z\left(\{z_I \}\right)$ is defined by
\[P_n^z\left(\{z_I \}\right)=\left \{ \bm{t}\in\R^n\mid  \sum_{i\in I}{t_i}\geq z_I \mbox{ for } I\neq [n], \mbox{ and } \sum_{i=1}^{n}{t_i}=z_{[n]}   \right \}. \]

There is a subclass of generalized permutahedra given by Minkowski sums of dilations of the faces of the standard $(n-1)$-simplex. For $I\subseteq [n]$, let $\Delta_I=\conv(\{e_i\mid  i\in I \})$, where $e_i$ is the $ith$ standard basis vector in $\R^n$ and $\Delta_{\emptyset}$ is the origin. Given a set $\{y_I\}_{I\subseteq [n]}$  of nonnegative real numbers with $y_\emptyset=0$, consider the polytope $\sum_{I\subseteq [n]}{y_I \Delta_I}$.

\begin{proposition}[\cite{genperms}, Proposition 6.3]
	\label{mobiusrelation}
	Given nonnegative real numbers $\{y_I\}_{I\subseteq [n]}$, set $z_I=\sum_{J\subseteq I}{y_J}$. Then 
	\[P_n^z\left(\{z_I \}\right) = \sum_{I\subseteq [n]}{y_I \Delta_I}. \]
\end{proposition}

\medskip

\noindent We now return to left-degree polynomials. Our goal is to show that 
	\[\mathrm{Newton}(L_{G,F}(\bm{t}))=P_n^z\left(\{z_I^F\}_{I\subseteq [n]}\right)\]
	for some parameters $\{z_I^F\}_{I \subseteq [n]}$. The proof relies on the following fact about flow polytopes, which readily follows from the  max-flow min-cut theorem.
	
	\begin{lemma}
		\label{feasibleflowproblem}
		Let $G$ be a graph on $[0,n]$ and $\bm{\alpha}=(\alpha_0,\ldots,\alpha_n)\in \mathbb{R}^{n+1}$. Then $\mathcal{F}_G(\bm{\alpha})$ is nonempty if and only if 		\begin{align}
			\label{existenceofaflow}
			\sum_{i=0}^{n}{\alpha_i}=0 \mbox{  and }
 \sum_{i\in S}{\alpha_i}\leq 0 \mbox{ for all } S\subseteq [0,n] \mbox{ with } \outdegG(S)=0.
		\end{align}
	\end{lemma}
	\begin{proof}
		Observe that the conditions \eqref{existenceofaflow} are necessary in order for $\mathcal{F}_G(\bm{\alpha})$ to be  nonempty. We now show they are also sufficient. For this, 
		we rephrase the problem as a max-flow problem on another graph. Let 
		\[G'=(V(G)\cup \{s,t\}, E(G)\cup \{(s, i) \mid i \in [0,n], \alpha_i>0)\}\cup \{(i, t) \mid i \in [0,n], \alpha_i<0)\}).\] Direct edges of $G'$ from smaller to larger vertices, where $s$ is the smallest and $t$ is the largest. 
		
		Let the edges $\{(s,i)\mid i \in [0,n],\, \alpha_i>0\}$ have upper capacity $\alpha_i$, and the edges $\{(i, t) \mid i \in [0,n],\, \alpha_i<0\}$,  have upper capacity $-\alpha_i$. Let the edges belonging to both $G$ and $G'$ have the  upper capacity $\sum_{i \in [0,n], \alpha_i>0}\alpha_i$. Assign all edges of $G'$ the lower capacity of $0$.

		If the maximum	flow on $G'$ saturates the edges incident to $s$ (equivalently, to $t$), then $\mathcal{F}_G(\bm{\alpha})$ is nonempty.  We thus proceed to show that if $\bm{\alpha}$ satisfies  \eqref{existenceofaflow} with the given $G$, then the maximum	flow on $G'$ saturates the edges incident to $s$. In other words, if $\bm{\alpha}$ satisfies  \eqref{existenceofaflow} with the given $G$, then the value of the maximum flow on $G'$ is  $\sum_{i \in [0,n], \alpha_i>0}\alpha_i$.
		
		Recall that by the max-flow min-cut theorem \cite[Theorem 10.3]{schrijver} the maximum value of an $s-t$ flow on $G'$ subject to the above capacity constraints equals the minimum capacity of an $s-t$ cut in $G'$. For the cut $(\{s\}, V(G)\backslash \{s\})$  the capacity is $\sum_{i \in [0,n], \alpha_i>0}\alpha_i$, and we show that this is the minimum capacity of an $s-t$  cut in $G'$.  If the cut contains any edge not incident to $s$ or $t$, then the capacity of that edge is already 	$\sum_{i \in [0,n], \alpha_i>0}\alpha_i$. 
		
		On the other hand, if the cut does not contain any edge not incident to $s$ or $t$, the partition of vertices is of the form $(\{s\}\cup S, S^c\cup \{t\}),$ where $S\subseteq [0,n]$ with   $\outdegG(S)=0$ and $S^c=[0,n]\backslash S$. Thus, by  \eqref{existenceofaflow} we have $\sum_{i\in S}{\alpha_i}\leq 0$. The capacity of the cut  $(\{s\}\cup S, S^c\cup \{t\})$ is \[\sum_{i \in S^c, (s,i) \in G'}\alpha_i-\sum_{i \in S, (i,t) \in G'} \alpha_i.\] Note that 
		\[0\geq\sum_{i\in S}{\alpha_i}= \sum_{i \in S, \alpha_i>0}\alpha_i+\sum_{i \in S, (i,t) \in G'} \alpha_i.\]
		Consequently,
		\begin{align*}
			\sum_{i \in S^c, (s,i) \in G'}\alpha_i-\sum_{i \in S, (i,t) \in G'} \alpha_i &\geq \sum_{i \in S^c, (s,i) \in G'}\alpha_i+\sum_{i \in S, \alpha_i>0}\alpha_i\\
			&=\sum_{i \in [0,n], \alpha_i>0}\alpha_i
		\end{align*}
		 In other words, the capacity of any cut is at least $\sum_{i \in [0,n], \alpha_i>0}\alpha_i$, and we saw that this is achieved. Thus, the value of the maximum	flow on $G'$ is  $\sum_{i \in [0,n], \alpha_i>0}\alpha_i$, as desired.
	\end{proof}	

	For $F\subseteq E(G\backslash 0)$, recall the numbers $f_{ij}$ given by 
	\[f_{ij}=\#\{(j,k)\in F\mid  k\leq i\}. \]
	By Corollary \ref{flows} (Theorem $\ref{finalflows}$ for the general case), $\LD(G,F)$ is in bijection with integral flows on the graph $\widetilde{G}\backslash\{s,0\}$ with the netflow vector $\bm{b}_G^F$ defined by
	\begin{align*}
	\bm{b}_G^F= (\indegG(1)-\outdegF(1),\dots,\indegG(n)-\outdegF(n),-\#E(G\backslash F))
	\end{align*} 
	via projection onto the edges $(i,t)$. 
	
	\begin{definition}
		For a collection of vertices $I$ of a graph $G$, define the outdegree $\outdegG(I)$ to be the number of edges from vertices in $I$ to vertices not in $I$.
	\end{definition}
	
	To each $I\subseteq[n]$, associate the integer $z_I^F$ given by 
%	\begin{align}
%	\label{parameters}
%	z_I^F=\min \left \{\sum_{i\in I}f(i,t)\mid  f \mbox{ is a flow on } \widetilde{G}\backslash\{s,0\} \mbox{ with netflow vector $\bm{b}_G^F$}  \right \}.
%	\end{align}
	
	\begin{align}
	\label{parameters}
	z_I^F=\sum_{i\in S}\indegG(i)-\outdegF(i)
%	
%	\min \left \{\sum_{i\in I}f(i,t)\mid  f \mbox{ is a flow on } \widetilde{G}\backslash\{s,0\} \mbox{ with netflow vector $\bm{b}_G^F$}  \right \}.
	\end{align}
	where $S\subseteq I$ is the maximal subset with $\outdeg_G(S)=0$.
	
	\begin{theorem}
		\label{genpermtheoremzI}
		For a simple graph $G$, $F\subseteq E(G\backslash 0)$, and $\{z_I^F\}$ the parameters defined by (\ref{parameters}), $\mathrm{Newton}(L_{G,F}(\bm{t}))$ is the generalized permutahedron
		\[\mathrm{Newton}(L_{G,F}(\bm{t}))=\conv(\LD(G,F))=P_n^z\left(\{z_I^F\}_{I\subseteq [n]}\right).\]
		Furthermore, each integer point of $P_n^z\left(\{z_I^F\}\right)$ is in $\LD(G,F)$, so $L_{G,F}(\bm{t})$ has SNP.
	\end{theorem}
	
	\begin{proof}
		First, it is easy to check that the parameters $z_I^F$ satisfy the supermodularity condition. Thus, $P_n^z\left(\{z_I^F\}_{I\subseteq [n]}\right)$ is a generalized permutahedron. To observe that $\conv(\LD(G,F))\subseteq P_n^z\left(\{z_I^F\}\right)$, simply recall that $\LD(G,F)$ equals the projection of integral flows on $\widetilde{G}\backslash\{s,0\}$ with netflow $\bm{b}_G^F$ onto the edges $\{(j,t)\}_{j\in[n]}$.
		
		\noindent For the reverse direction, let $\bm{d}$ denote the truncation of $\bm{b}_G^F$ by its last entry, that is let $\bm{d}=(d_1,\dots,d_n)$ where \[d_i=\indegG(i)-\outdegF(i).\]
		We must show that each point $\bm{x}=(x_1,\ldots,x_n)\in P_n^z\left(\{z_I^F\}\right)$, the assignment $a_{nj}=x_j$ in $\widetilde{G}\backslash\{s,0\}$ can be extended to a flow on $\widetilde{G}\backslash\{s,0\}$. This is equivalent to showing  \[\mathcal{F}_{G\backslash 0}(\bm{d}-\bm{x})\neq \emptyset.\]
		By Lemma \ref{feasibleflowproblem}, it suffices to note that
		\[\sum_{i\in S}{d_i-x_i}\leq 0 \mbox{ for all } S\subseteq [n] \mbox{ with } \outdegG(S)=0.\]
		However, since $\outdeg_G(S)=0$, we have
		\[\sum_{i\in S}{x_i}\geq z_S= \sum_{i\in S}{d_i}. \]
	\end{proof}
	
	We further show that $\mathrm{Newton}(L_{G,F}(\bm{t}))$ can be written as $\sum_{I\subseteq [n]}{y_I \Delta_I}$ for some parameters $y_I$. Let $L=\{J \subseteq [n]\mid  \outdeg_G(J)=0 \}$. $L$ is a lattice under union and intersection, so consider the set $Q$ of join-irreducible elements of $L$ (elements that cannot be written as the union of other elements). 
	
	We explicitly describe the members of $Q$. Let $\delta(i)$ denote all the vertices of $G$ that can be reached from $i$ by a directed path (including $i$ itself).	
	\begin{lemma}
		An element $J\in L$ is join-irreducible if and only if $J=\delta(i)$ for some $i\in [n]$.
	\end{lemma}
	
	For $J\subseteq [n]$, define
	\begin{align}
		\label{convLDtypeY}
		y_J^F=\begin{cases}
		\indegG(k)-\outdegF(k) &\mbox{ if } J\in Q \mbox{, $J$ covers $J'$ in $L,$}\,\, J\backslash J'=\{k\}\\
		0 &\mbox{ if } J\notin Q
		\end{cases} 
	\end{align}
	
	\begin{proposition}
		\label{typey}
		For any simple graph $G$ and $F\subseteq E(G\backslash 0)$,
		\[P_n^z\left(\{z_I^F\}\right)=\sum_{I\subseteq [n]}{y_J^F \Delta_I}. \]
	\end{proposition}
	
	\begin{proof}
		Note that $z_I^F=z_{I_1}^F$ where $I_1$ is the largest element of $L$ contained in $I$. Thus, 
		\[z_I^F=z_{I_1}^F= \sum_{k\in I_1}{b_k} = \sum_{\mathclap{\substack{J\in Q\\  J\subseteq I_1}}}{y_J^F} =\sum_{J\subseteq I}{y_J^F}.  \]
		Apply Proposition \ref{mobiusrelation}.
	\end{proof}
	
	From (\ref{convLDtypeY}), we can read off the $\{y_I^F\}$ decomposition of $\mathrm{Newton}(L_{G,F}(\bm{t}))$.  Then, 
	\begin{align}
		\label{ldpolytopeydescription}
		\mathrm{Newton}(L_{G,F}(\bm{t}))= \sum_{i=1}^n{(\indegG(i)-\outdegF(i)) \Delta_{\delta(i)} }.
	\end{align}
	
	\begin{example}
		For a simple graph $G$, recall that the transitive closure of $G$ is the simple graph formed by adding edges $(i,j)$ to $E(G)$ whenever the vertices $i\neq j$ are connected by a directed path in $G$. If $G$ is a simple graph on $[0,n]$ such that the transitive closure of $G\backslash\{0\}$ is complete, then for each $F\subseteq E(G\backslash 0)$, 
		\[\mathrm{Newton}(L_{G,F}(\bm{t})) = \Pi_n\left (\indegG(1)-\outdegF(1), \ldots, \indegG(n)-\outdegF(n)\right )
		\]
		
		\noindent where $\Pi_n(\bm{x})$ is the Pitman-Stanley polytope as defined in \cite{sppolytope}, but shifted up one dimension in affine space, that is
		\begin{align*}
		\Pi_n(\bm{x})&=\left \{\bm{t}\in \R^n_{\geq 0}\mid  \sum_{p=1}^{k}{t_p} \leq \sum_{p=1}^{k}{x_p} \mbox{ for } k\in [n-1],\, \mbox{ and } \sum_{p=1}^{n}{t_p} = \sum_{p=1}^{n}{x_p}  \right \} \\
		&=x_n\Delta_{\{n\}}+x_{n-1}\Delta_{\{n-1,n\}}+\cdots+x_1\Delta_{[n]}.
		\end{align*}
	\end{example}
	
	\begin{proposition}
		If $T$ is a tree on $[0,n]$, then $\mathrm{Newton}(L_{T,F}(\bm{t}))$ is a simple polytope.
	\end{proposition}
	
	\begin{proof}
		By the Cone-Preposet Dictionary for generalized permutahedra, (\cite{genpermfaces}, Proposition 3.5) it is enough to show that each vertex poset $Q_v$ is a tree-poset, that is, its Hasse diagram has no cycles. To show this, let $I\subseteq [n]$ and consider the normal fan $N(\Delta_I)$ of the simplex $\Delta_I$. By (\ref{ldpolytopeydescription}), the normal fan of $\mathrm{Newton}(L_{G,F}(\bm{t}))$ is the refinement of normal fans $N(\Delta_I)$.
		
		Thus, a maximal cone of the normal fan of $\mathrm{Newton}(L_{G,F}(\bm{t}))$ is given by an intersection of maximal cones in each $N(\Delta_I)$ for $I=\delta(j)$, $j\in [n]$, $\indeg_T(j)>0$. A maximal cone in $N(\Delta_I)$ gives the vertex poset relations $x_i>x_j$ for all $j\in I$ and any chosen $i\in I$. Thus, relations in the Hasse diagram of a vertex poset lift to undirected paths in $T$.
		
		If some $Q_v$ has a cycle $C$, then we can lift the relations to get two different paths in $T$ between two vertices. This subgraph will contain a cycle, contradicting that $T$ is a tree.
	\end{proof}

	The Newton polytopes of the homogeneous components of $L_G(\bm{t})$ are also generalized permutahedra. 
	
	\begin{definition}
		For each $k\geq 0$ let $L_G^k(\bm{t})$ denote the degree $\#E(G)-k$ homogeneous component of $L_G(\bm{t})$, that is 
		\begin{align*}
		L_G^k(\bm{t})=\sum_{\mathclap{\substack{F\subseteq E(G\backslash 0)\\
					 \#F =k}}}{L_{G,F}(\bm{t})}
		\end{align*}
	\end{definition}
	
	For a simple graph $G$ on $[0,n]$, Theorem \ref{cap} showed that the augmented graph $\Gaug$ of Definition \ref{moddingofG} has the property that the projection of integral flows on $\Gaug$ with netflow
	\[\bm{b}_G^{\emptyset}=\left (\indegG(1),\ldots, \indegG(n),-\#E(G) \right )\]  
	and capacitance $0\leq y_{ij}\leq 1 $ for all $1\leq j<i\leq n$ onto the edges labeled $a_{nj}$ for $j \in [n]$ is exactly $\LD(G)$. The following construction is a variation on this theme designed so its integral flows will only project to left-degree sequences whose entries have a particular sum.
	\begin{definition}
		Given a simple graph $G$ on $[0,n]$ and $k\geq 0$, let $G^{(k)}$ be the graph on $[1,n+1]\cup \{t\}$ with labeled edges $E_a\cup E_z\cup E_y$ where
		\begin{align*}
		E_a&\mbox{ consists of edges }a_{nj}:j\to t\mbox{ for } j\in [n];  \\
		E_z&\mbox{ consists of edges }z_{ij}:j\to i\mbox{ for }(j,i)\in E(G\backslash 0);\\
		E_y&\mbox{ consists of edges }y_{ij}:j \to n+1\mbox{ for } (j,i)\in E(G\backslash 0).
		\end{align*}
		
		The flow polytope $\F_{G^{(k)}}^c(\bm{b}_G^{(k)} )$ is the flow polytope of 
		 $G^{(k)}$ with  netflow vector $\bm{b}_G^{(k)} = (\indegG(1),\ldots, \indegG(n), -k, k-\#E(G))$ and capacities $1$ on the edges $y_{ij}$.
	\end{definition}
	\begin{example}
		For $G$ the complete graph on $[0,3]$, $G^{(k)}$ is shown below alongside $\Gaug$ for comparison.\\
		\begin{center}
			\includegraphics[]{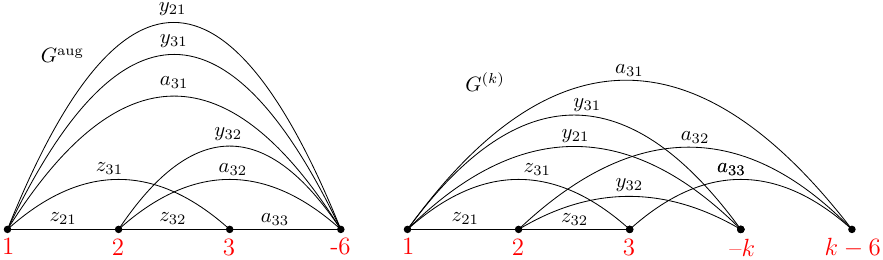}
		\end{center}
	\end{example}
	
	Note that capacitated integral flows on $G^{(k)}$ with netflow $\bm{b}_G^{(k)}$ are in bijection with capacitated integral flows on $\Gaug$ with netflow $\bm{b}_G^\emptyset$ where exactly $k$ edges $y_{ij}$ have flow 1, and the bijection preserves the values on the edges $\{a_{nj}\mid j\in[n] \}$.
	
	\begin{theorem}
		\label{ldpolynomialhomogeneousnewton}
		For $k\geq 0$, if $\psi$ is the projection that takes a flow on $\mathcal{F}_{G^{(k)}}^c\left (\bm{b}_G^{(k)}\right )$ to the tuple of its values on the edges labeled $a_{nj}$ for $j$ in $[n]$, then 
		\begin{align*}
		\mathrm{Newton}\left(L^k_G(\bm{t})\right)=\psi \left ( \mathcal{F}_{G^{(k)}}^c\left (\bm{b}_G^{(k)}\right ) \right ).
		\end{align*}
%		Furthermore, each integer point in the right-hand side is a left-degree sequence with components that sum to $\#E(G)-k$, so $L_G^k$ has SNP.
	\end{theorem}
	\begin{proof}
		Let $\alpha$ be an integer point in $\mathrm{Newton}\left(L^k_G(\bm{t})\right)$, so $\alpha \in \LD(G,F)$ for $F\subseteq E(G\backslash 0)$ with $\#F=k$. Then, $\alpha$ corresponds to a capacitated integral flow on $\Gaug$ with netflow $\bm{b}_G^\emptyset$, which in turn corresponds to a capacitated integral flow on $G^{(k)}$ with netflow $\bm{b}_G^{(k)}$ that $\psi$ takes to $\alpha$.
		
		Conversely, let $\alpha$ be an integer point in $\psi \left ( \mathcal{F}_{G^{(k)}}^c\left (\bm{b}_G^{(k)}\right ) \right )$. Lift $\alpha$ to an integral flow $f$ on $G^{(k)}$. The flow $f$ corresponds to an integral flow on $\Gaug$, so if $F=\{(j,i)\mid  y_{ij}=1 \mbox{ in } f \}$, then $\#F=k$ and $\alpha\in\LD(G,F)$. 
	\end{proof}
	
	Similar to the proof of Theorem \ref{genpermtheoremzI}, for $k\geq 0$ and $I\subseteq [n]$, define parameters $z_I^{(k)}$ by
	\begin{align}
		\label{homogeneousparameters}
		z_I^{(k)}=\min \left \{\sum_{i\in I}f(i,t)\mid  f \mbox{ is a flow on } G^{(k)} \mbox{ with netflow vector $\bm{b}_G^{(k)}$}  \right \}. 	
	\end{align}

	\begin{theorem}
		\label{ldhomogeneouspiecesnewton}
		For $k\geq 0$ and $\{z_I^{(k)}\}$ the parameters defined by (\ref{homogeneousparameters}), $\mathrm{Newton}(L_G^k(\bm{t}))$ is the generalized permutahedron
		\[\mathrm{Newton}(L_G^k(\bm{t}))=P_n^z\left(\{z_I^{(k)}\}_{I\subseteq [n]}\right).\]
		Furthermore, each integer point of $P_n^z\left(\{z_I^{(k)}\}\right)$ is a left-degree sequence, so $L_{G}^k(\bm{t})$ has SNP. Additionally, if $G$ is a tree, then $L_G^0(\bm{t})$ is the integer point transform of its Newton polytope.
	\end{theorem}
	\begin{proof}
		The proof of the first two statements is analogous to that of Theorem \ref{genpermtheoremzI}. Alternatively, SNP follows from the fact that the $\mathrm{Newton}(L^k_G)$ is the intersection of $\mathrm{Newton}(L_G)$ by a hyperplane.
		
		Recall that the integer point transform of a polytope $P\subseteq \mathbb{R}^m$ is the polynomial
		\[
		L_P(x_1, \ldots, x_m)=
		\sum_{p \in P\cap \mathbb{Z}^m} \bm{x}^{p}.
		\]
		To prove the third statement we must show that if $G$ is a tree, all nonzero coefficients of  $L_G^0$ are 1. It follows from Corollary \ref{flows} (Theorem \ref{finalflows}) that $\LD(G,\emptyset)$ equals the multiset of projections of integral flows on $\widetilde{G}\backslash\{s,0\}$ with the netflow vector $\bm{b}_G^\emptyset$. Then, the multiplicity of any particular $\alpha\in \LD(T,\emptyset)$ is the number of flows on $G\backslash 0$ with netflow  $\bm{b}_G^\emptyset-\alpha$. However, trees admit at most one flow for any given netflow vector, so every element of $\LD(G,\emptyset)$ has multiplicity 1. This implies all coefficients in $L_G^0$ are 0 or 1.
	\end{proof}
	
	Theorems \ref{cap} and  \ref{ldhomogeneouspiecesnewton} imply the following.
	
	\begin{corollary}
		\label{hyperpl}
		Given a graph $G$ on the vertex set $[0,n]$ with $m$ edges, we have that 
		\[\mathrm{Newton}(L_G(\bm{t})) \cap \left\{(x_1, \ldots, x_n)\in \mathbb{R}^n \mid \sum_{i=1}^n x_i=m-k\right\}=P_n^z\left\{z_I^{(k)}\right\}_{I\subseteq [n]},\]
		for the parameters $\{z_I^{(k)}\}$ given in (\ref{homogeneousparameters}).
	\end{corollary}
	\begin{proof} 
		We have that $\mathrm{Newton}(L_G(\bm{t})) \cap \{(x_1, \ldots, x_n)\in \mathbb{R}^n \mid \sum_{i=1}^n=m-k\}=\mathrm{Newton}(L_G^k(\bm{t})),$ which by Theorem \ref{ldhomogeneouspiecesnewton}  equals $P_n^z\left(\{z_I^{(k)}\}_{I\subseteq [n]}\right)$.
	\end{proof}
	
 Theorems \ref{pstheoremaltproof} and \ref{ldhomogeneouspiecesnewton} imply:

\begin{corollary}
		\label{latticeptenum}
		If $G$ is a tree on $[0,n]$, then the normalized volume of the flow polytope of $\widetilde{G}$ is  
		\begin{align*}
			\mathrm{Vol} \,\,\mathcal{F}_{\widetilde{G}} = \operatorname{Ehr}(P_G^0, 1), 
		\end{align*}
		where $P_G^0=\mathrm{Newton}(L_G^0(\bm{t}))$ is the generalized permutahedron specified in Theorem  \ref{ldhomogeneouspiecesnewton}.
		
	\end{corollary}
	
	Corollary \ref{latticeptenum} is of the same flavor as Postnikov's following beautiful result; for the details of the terminology used in this theorem refer to \cite{genperms}.
	
	\begin{theorem}\cite[Theorem 12.9]{genperms} \label{thm:post} For a bipartite graph $G$, the normalized volume of the root polytope $Q_G$ is 
	\begin{equation*} 
		\label{eq:Qvol} \mathrm{Vol}\, Q_G = \operatorname{Ehr}(P_G^{-}, 1), 
	\end{equation*}
where $P_G^{-}$ is the trimmed generalized permutahedron.
	
	\end{theorem}

Root polytopes and flow polytopes are closely related, as can be seen by contrasting the techniques and results in the papers \cite{root1, root2, prod,  mm, genperms}. It is thus reasonable to expect that Corollary  \ref{latticeptenum} and Theorem \ref{thm:post} are related mathematically. We invite the interested reader to investigate their relationship.

\section{Newton polytopes of Schubert and Grothendieck polynomials}
\label{sec5}
	In this section, we discuss the connection between left-degree sequences,  Schubert polynomials, and Grothendieck polynomials discovered in \cite{pipe1} and relate it to their Newton polytopes. Our main theorem is the following.
	
	\begin{letteredtheorem}
		\label{theoremC}
		Let $\pi\in S_{n+1}$ be of the form $\pi=1\pi'$ where $\pi'$ is a dominant permutation of $\{2,3,\ldots n+1\}.$ Then the Grothendieck polynomial $\mathfrak{G}_{\pi}$ has SNP and the Newton polytope of each homogeneous component of $\mathfrak{G}_{\pi}$ is a generalized permutahedron. In particular, the Schubert polynomial $\mathfrak{S}_{\pi}$ has SNP and $\mathrm{Newton}(\mathfrak{S}_{\pi})$ is a generalized permutahedron. Moreover, $\mathfrak{S}_{\pi}$ is the integer point transform of its Newton polytope.
	\end{letteredtheorem}
	
	Theorem \ref{theoremC} implies that the recent conjectures of  Monical, Tokcan, and Yong \cite[Conjecture 5.1 \& 5.5]{MTY} are true in the special case of permutations  $1\pi'$, where $\pi'$ is a dominant permutation. The authors and Alex Fink prove \cite[Conjecture 5.1]{MTY} in its full generality in \cite{FMS}. The following conjecture,  discovered jointly with Alex Fink, is a strengthening of \cite[Conjecture 5.5]{MTY} based on the results of this paper. We have tested it for all $\pi\in S_n$, for $n\leq 8$.
		
	\begin{conjecture}
		\label{conj:groth}
		The Grothendieck polynomial $\mathfrak{G}_{\pi}$ has SNP and the Newton polytope of each homogeneous component of $\mathfrak{G}_{\pi}$ is a generalized permutahedron.
	\end{conjecture}  

	Since   \cite{pipe1}  uses right-degree sequences and right-degree polynomials instead of their left-degree counterparts, we will adopt this convention throughout this section. To simplify notation, all graphs in this section will be on the vertex set  $[n+1]$. Note the following easy relation between right-degree and left-degree. 
	
	Given a graph $G$ on vertex set $[n+1]$, let $G^*$ be the mirror image of the graph $G$ with vertex set shifted to $[0,n]$. More formally, let $G^*$ be the graph on vertices $[0,n]$ with edges
	\[E(G^*)=\{(n+1-j,n+1-i) \mid  (i,j) \in E(G) \}.\]
	The right-degree sequences of $G$ are exactly the left-degree sequences of $G^*$ read backwards. Via Theorem \ref{theoremA} of Section \ref{sec6} in hand, we define the \newword{right-degree multiset} $\RD(G)$ as the multiset of right-degree sequences of leaves in any reduction tree of $G$, and $\RD(G,\emptyset)$ the submultiset of sequences whose components sum to $\#E(G)$ (notation consistent with $\LD(G,F)$ in Definition \ref{ldsequencesfromF}).
	
	\begin{definition}
		\label{def:RG}
		For any graph $G$ on $[n+1]$, define the \newword{right-degree polynomial} $R_G$ by
		\[
			R_G(t_1,t_2,\ldots t_{n}) = L_{G^*}(t_{n}, t_{n-1}, \ldots, t_1)=\sum_{\alpha\in \RD(G)} (-1)^{\mathrm{codim}(\alpha)} t_1^{\alpha_1} t_2^{\alpha_2} \ldots t_{n}^{\alpha_{n}}
		\] 
		where $\mathrm{codim}(\alpha)=\#E(G)-\sum_{i=1}^{n}{\alpha_i}$.

		\noindent For $k\geq 0$, let $R_G^k(\bm{t})$ denote
		 the degree $\#E(G)-k$ homogeneous component of $R_G(\bm{t})$.
		
		Define the reduced right-degree polynomial $\widetilde{R}_G$ as follows: If $\{v_{i_1},\ldots v_{i_k}  \}$ are the vertices of $G$ with positive outdegree, then $R_G$ is a polynomial in $t_{i_1},\ldots, t_{i_{k}}$. Obtain $\widetilde{R}_G$ by relabeling the variables $t_{i_{m}}$ by $t_m$ for each $m$. Note that $R_G^0$ (resp. $\widetilde{R^0_G}$) is the top homogeneous component of $R_G$ (resp. $\widetilde{R}_G$), and is given by 
		\[R_G^0(t_1,\ldots,t_{n})=\sum_{\alpha\in \RD(G,\emptyset)}  t_1^{\alpha_1} t_2^{\alpha_2} \ldots t_{n}^{\alpha_{n}}\]
	\end{definition}
	                 
The following statement collects the right-degree analogues of Corollary \ref{cor:LGSNP} and Theorem \ref{ldhomogeneouspiecesnewton} from the previous section.
\begin{theorem}
	\label{rdanaloguecollection}
	Let $G$ be a graph on $[n+1]$. Then, $R_G(\bm{t})$ has SNP, and the Newton polytope of each homogeneous component $R_G^k$ is a generalized permutahedron. Additionally, if $G$ is a tree, then $R_G^0(\bm{t})$ equals the integer point transform of its Newton polytope.
\end{theorem}

We now recall the definition of pipe dreams of a permutation and the characterization of Schubert and Grothendieck polynomials in terms of pipe dreams.
\begin{definition}
A \newword{pipe dream} for $\pi\in S_{n+1}$ is a tiling of an $(n+1)\times (n+1)$ matrix with two tiles, crosses $\textcross$ and elbows $\textelbow$, such that
\begin{enumerate}
	\item all tiles in the weak south-east triangle are elbows, and
	\item if we write $1,2,\ldots, n+1$ on the top and follow the strands (ignoring second crossings among the same strands), they come out on the left and read $\pi$ from top to bottom.
\end{enumerate}
A pipe dream is \newword{reduced} if no two strands cross twice.
\end{definition}

\begin{figure}[ht]
	\includegraphics[scale=.4]{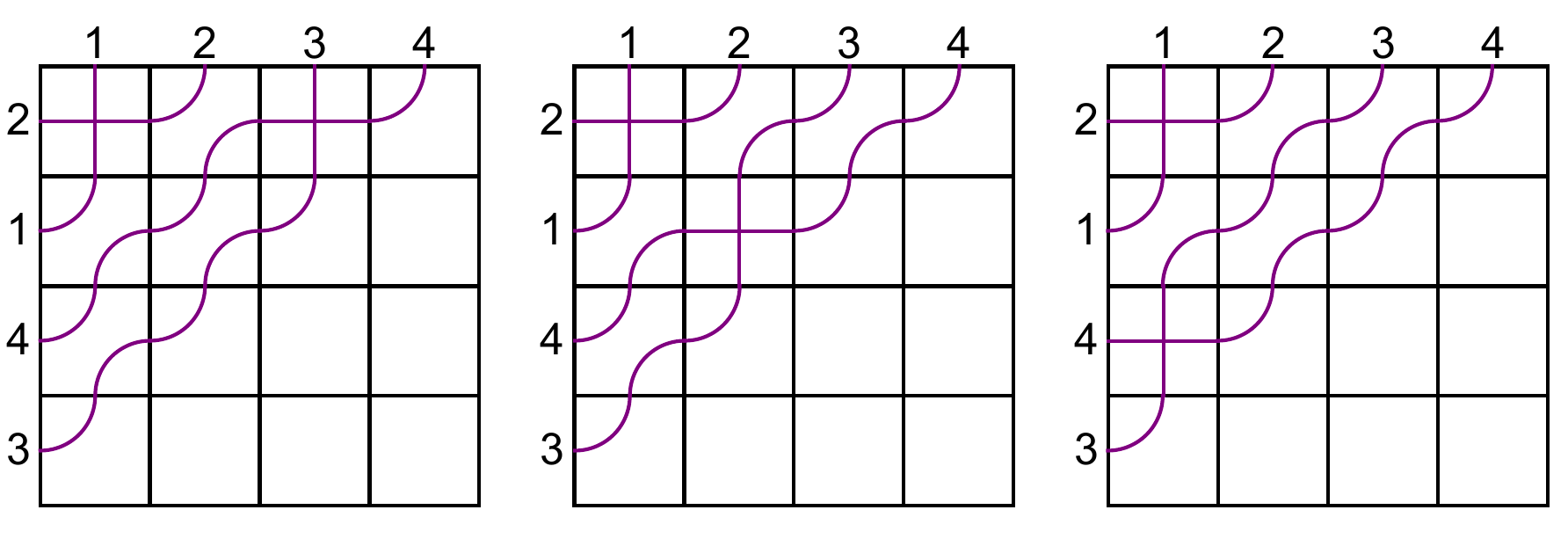}
	\caption{The reduced pipe dreams for $\pi=2143$. All tiles not shown are elbows.}
\end{figure}

	For $\pi\in S_{n+1}$ let $\mathrm{PD}(\pi)$ denote the collection of all pipe dreams of $\pi$ and $\mathrm{RPD}(\pi)$ the collection of all reduced pipe dreams of $\pi$. For $P\in \mathrm{PD}(\pi)$, define the weight of $P$ by
	\[wt(P)=\prod_{(i,j)\in \mathrm{cross}(P)}t_i\]
	where $\mathrm{cross}(P)$ denotes the set of indices of all crosses in $P$.
	
	Recall that for any $\pi\in S_{n+1}$, the Grothendieck polynomial $\mathfrak{G}_\pi$ can be represented in terms of pipe dreams of $\pi$ by
	\begin{align*}
		\mathfrak{G}_\pi(t_1,\ldots, t_{n}) = \sum_{P\in \mathrm{PD}(\pi)}{wt(P)}, 
	\end{align*}
	and the Schubert polynomial $\mathfrak{S}_\pi$ is the lowest degree homogeneous component of the Grothendieck polynomial:
	\[\mathfrak{S}_\pi(t_1,\ldots,t_{n})=\sum_{P\in \mathrm{RPD}(\pi)}{wt(P)}.\]
	
	In \cite[Theorem 5.1]{pipe1}, it is proved that for any noncrossing tree $T$, the right-degree sequences $\RD(T)$ (see paragraph preceding Definition \ref{def:RG}) are independent of the choice of reduction tree for $T$, and the following connection to Grothendieck polynomials is shown.
	\begin{theorem}[{\cite[Theorem 5.3]{pipe1}}]
		\label{relatingrdandschub}
		Let $\pi\in S_{n+1}$ be of the form $\pi=1\pi'$ where $\pi'$ is a dominant permutation of $\{2,3,\ldots n+1\}.$ Then, there is a tree $T(\pi)$ and nonnegative integers $g_i=g_i(\pi)$ such that 
		\[\widetilde{R}_{T(\pi)}(\bm{t}) =\left (\prod_{i=1}^{n}{t_i^{g_i}} \right )\mathfrak{G}_{\pi}(t_1^{-1},\ldots, t_{n}^{-1}). \]
		
		\noindent Explicitly, if $C(\pi)$ denotes the set $\mathrm{core}(\pi)\cup \{(1,1)\}$, then $g_i(\pi)$ is the number of boxes in column $i$ of $C(\pi)$.
	\end{theorem}
	
	\noindent In terms of Newton polytopes, Theorem \ref{relatingrdandschub} implies  
	\begin{align*}
		\mathrm{Newton}\left(\mathfrak{G}_{\pi}\right )=\varphi\left (\mathrm{Newton}\left(\widetilde{R}_{T\left(\pi\right )}\left(\bm{t}\right )\right)\right )
	\end{align*}
	and 
	\begin{align*}
	\mathrm{Newton}\left(\mathfrak{S}_{\pi}\right )=\varphi\left (\mathrm{Newton}\left(\widetilde{R}^0_{T\left(\pi\right )}\left(\bm{t}\right )\right)\right )
	\end{align*}
	\noindent where $\varphi$ is the integral equivalence
	\[(x_1,\ldots,x_{n}) \mapsto \left (g_1-x_1,\ldots,g_{n}-x_{n}\right ).\]

	\begin{proof}[Proof of Theorem \ref{theoremC}]
		By Theorem \ref{rdanaloguecollection}, right-degree polynomials $R_G(\bm{t})$ have SNP. Since $\mathrm{Newton}\left(\widetilde{R}_{T\left(\pi\right )}\right )$ is the image of  $\mathrm{Newton}\left(R_{T\left(\pi\right )}\right )$ by a projection forgetting coordinates that are always zero, it follows from Theorem \ref{relatingrdandschub} that $\mathfrak{G}_{\pi}$ has SNP. 
		
		Theorem \ref{rdanaloguecollection} and Theorem \ref{relatingrdandschub} also yield that each homogeneous component of  $\mathfrak{G}_{\pi}$ has SNP and that their Newton polytopes are generalized permutahedra. In particular, this holds for the Schubert polynomial. Since by \cite{pipe1} the Schubert polynomial of $\pi=1\pi'$, where $\pi'$ is a dominant permutation, has $0,1$ coefficients, the last statement also follows.
	\end{proof}
	
	From the proof of Theorem \ref{relatingrdandschub} in \cite{pipe1}, one can   infer the following new transition rule for Schubert polynomials of permutations of the form $1\pi'$ with $\pi'$ dominant.  
	
	\begin{lemma}[Transition rule for $1\pi'$ Schubert polynomials] Let $\pi\in S_{n+1}$ be of the form $\pi=1\pi'$ with $\pi'$ a dominant permutation of $\{2,\ldots,n+1\}$. Let $\pi'$ have diagram given by the partition $\lambda(\pi')=(\lambda_1,\cdots, \lambda_z)$ with $\lambda_z=k$. For $0\leq l \leq k$, let $w_l$ be the permutation on $\{2,\ldots,n+1\}$ whose diagram is the partition  $(\lambda_1-(k-l), \ldots, \lambda_{z-1}-(k-l))$. Then
		\[
		\mathfrak{S}_{\pi}(\bm{x})= \sum_{l=0}^{k}{\left (\prod_{m=1}^{l}{x_m}  \right ) \left (\prod_{p=l+2}^{k+1}{x_p^z}  \right ) \mathfrak{S}_{1w_l}(\bm{x}_{\phi_l})}
		\]
		where $\bm{x}=(x_1,x_2,\ldots)$, $\bm{x}_{\phi_l}=(x_{\phi_l(1)},x_{\phi_l(2)},\ldots)$, and $\phi_l(i)=\begin{cases}
		i &\mbox{ if } i\leq l+1\\
		i+k-l &\mbox{ if } i\geq l+2
		\end{cases}$
	\end{lemma}

	\begin{example}
		Let $\pi=14523$. Then, $\pi'=4523$, so $\lambda(\pi')=(2,2)$. For $0\leq l\leq 2$,the permutation $w_l$ will have diagram given by the partition $(l)$. These permutations are $w_0=2345$, $w_1=3245$, and $w_2=3425$. 
		Hence, the terms in the transition rule are
		\begin{align*}
			&(1)(x_2^2x_3^2)\mathfrak{S}_{1w_0}(x_1,x_4,x_5,x_6) = x_2^2x_3^2 \\
			&(x_1)(x_3^2)\mathfrak{S}_{1w_1}(x_1,x_2,x_4,x_5) = x_1^2 x_3^2  + x_1 x_2 x_3^2\\
			&(x_1x_2)(1)\mathfrak{S}_{1w_2}(x_1,x_2,x_3,x_4) = x_1^2  x_2^2 + x_1^2  x_2 x_3 + x_1 x_2^2  x_3.
		\end{align*}
		
		\noindent Adding these terms together gives the expected polynomial
		\begin{align*}
			\mathfrak{S}_{\pi}(x_1,x_2,x_3,x_4) = x_1^2  x_2^2  + x_1^2  x_2 x_3 + x_1 x_2^2  x_3 + x_1^2  x_3^2  + x_1 x_2 x_3^2  + x_2^2  x_3^2.
		\end{align*}
	\end{example}
	
	\section{Left-degree sequences as invariants}
	\label{sec6}
	In this section we prove the results of Section \ref{sec3} without the assumption that $G$ is simple. Similar adjustments can be made to generalize Sections \ref{sec4} and \ref{sec5} away from simple graphs. In this generality, we also prove the following main result.
	\begin{namedtheorem}[A]
		Let $G$ be any graph on $[0,n]$. Then for any reduction tree $\mathcal{R}$ of $G$,
		\[\mathrm{LD}(G)=\inseq(\mathcal{R}).\]
	\end{namedtheorem}
	Theorem \ref{theoremA} was first proved independently by Grinberg \cite{grinberg}. 	
	To deal with multiple edges in $E(G)$, we view each element of $E(G)$ as being distinct. Formally, we may think of assigning a distinguishing number to each copy of a multiple edge. In this way, we may speak of subsets $F\subseteq E(G\backslash 0)$ in the usual sense. 	
		
	For $G$ any graph on the vertex set $[0,n]$, we can still construct the reduction tree $\mathcal{T}(G)$ using the same algorithm as before in Definition \ref{specialreductiontree}. As in the case of simple graphs, the leaves of this specific reduction tree can be encoded as solutions to some constraint arrays. The key is using a generalized version of Lemma \ref{cornerstonelemma} with multiple incoming and outgoing edges at vertex $v$. This generalization is derived the same way and is not harder, but far more technical. The arrays we obtain are no longer necessarily triangular, but rather they may be staggered. This is explained below and demonstrated in Examples \ref{generalizedtriarrayexample} and \ref{ex:F}. We leave the proofs to the interested reader; they are straightforward generalizations of those in the previous section. With $\mathcal{T}(G)$ in hand, $\LD(G)$ is defined exactly as before.
	
%	\newword{Triangular arrays $\Tri_G(\emptyset)$ for arbitrary $G$.}
	We now describe how to define the arrays $\Tri_G(\emptyset)$. Start with the array constructed for simple graphs in Definition \ref{def:triGempty}. Replace each $a_{ij}$ by $a_{ij}^{(1)}$ in Definition \ref{def:triGempty} and Theorem \ref{arrayconstraints}. Add variables $a_{ij}^{(k)}$ with $k>1$ for each additional copy of the edge $(j,i)$ appearing in $G$. When there are $k>1$ copies of the edge $(j,i)\in E(G)$, also replace $a^{(1)}_{ij}\leq a^{(1)}_{i-1,j}$ in the constraint array by $a^{(1)}_{ij}\leq a^{(2)}_{ij}\leq \cdots \leq a^{(k)}_{ij}\leq a^{(1)}_{i-1,j}$. The following example demonstrates these changes.
	
	\begin{example}
		\label{generalizedtriarrayexample}
		Following Example \ref{triangulararrayofconstraintsexample}, if $G$ is the graph on vertex set $[0,4]$ with \[E(G)=\{(0,1), (0,1), (0,2),(1,2),(1,2), (2,3),(2,4),(3,4),(3,4) \},\] we obtain the constraints
		\begin{align*}
				&0\leq a_{41}^{(1)}=a_{31}^{(1)}= a_{21}^{(1)}\leq a_{21}^{(2)}\leq a^{(1)}_{11}=2\\
				&0\leq a_{42}^{(1)}\leq a_{32}^{(1)}\leq a_{22}^{(1)}=5-a^{(1)}_{21}\\
				&0\leq a_{43}^{(1)}\leq a_{43}^{(2)} \leq a_{33}^{(1)}= 6-a_{31}^{(1)}-a_{32}^{(1)}\\
				&0\leq a_{44}^{(1)}=9-a_{41}^{(1)}-a_{42}^{(1)}-a_{43}^{(1)}
		\end{align*}
	\end{example}
	
	Defining $\Tri_G(F)$ for arbitrary $G$ is requires analogous modifications. View $E(G)$ as a multiset, so we formally view each copy of a multiple edge $(j,i)$ as a distinct element. Let $F$ vary over subsets of $E(G\backslash 0)$, and define $\Tri_G(F)$ from (the general version of) $\Tri_G(\emptyset)$  as before using the numbers $f_{ij}$ of (\ref{fijnumbers}) and treating each $a_{ij}^{(m)}$ identically for different $m$. 
	
	\begin{example} \label{ex:F}
		With $G$ as in Example \ref{generalizedtriarrayexample} and $F=\{(1,2),(1,2),(2,3)\}$, the array $\Tri(F)$ is given by 
		\begin{align*}
				&2\leq a_{41}^{(1)}+2=a_{31}^{(1)}+2= a_{21}^{(1)}+2\leq a_{21}^{(2)}+2\leq a^{(1)}_{11}=2\\
				&1\leq a_{42}^{(1)}+1\leq a_{32}^{(1)}+1\leq a_{22}^{(1)}=3-a^{(1)}_{21}\\
				&0\leq a_{43}^{(1)}\leq a_{43}^{(2)} \leq a_{33}^{(1)}= 3-a_{31}^{(1)}-a_{32}^{(1)}\\
				&0\leq a_{44}^{(1)}=6-a_{41}^{(1)}-a_{42}^{(1)}-a_{43}^{(1)}
		\end{align*}
	\end{example}
	Using the definition of $\Tri_G(F)$ for arbitrary graphs $G$, we can extend the definitions of $\Sol_G(F)$ and $\LD(G,F)$ from simple graphs to arbitrary graphs $G$.
	As in Proposition \ref{tapisflowpolytope}, for each $F\subseteq E(G\backslash 0)$ the polytope  $\poly(\Tri_G(F))$ is integrally equivalent to the flow polytope of a graph $\gr(G)$, a straightforward generalization of Definition \ref{tagsimplefinalversion}. The proofs of Theorem \ref{isomorphism} and Corollary \ref{cor:LGSNP} then go through with minor changes. In particular, we have the following summary result.

	\begin{theorem}
	\label{finalflows}
	Let $G$ be a graph on $[0,n]$, $\rho$ be the map that takes a triangular array in any $\Sol_G(F)$ to its first column $\left(a_{n1}^{(1)}, \ldots, a_{n n}^{(1)}\right)$, and $\psi$ be the map that takes a flow on $\widetilde{G}\backslash \{s,0\}$ to the tuple of its values on the edges $\{(j,t)\mid  j\in[n] \}$. For $F\subseteq E(G\backslash 0)$, recall the netflow vector 
	\[\bm{b}_G^F=\left (\indegG(1)-\outdegF(1), \ldots , \, \indegG(n)-\outdegF(n), -\#E(G\backslash F)   \right ).\]
	Then for each $F\subseteq E(G\backslash 0)$, 
	\[\LD(G,F)=\rho\left(\SolG(F) \right)=\psi\left(\mathcal{F}_{\widetilde{G}\backslash\{s,0\}}\left(\bm{b}_G^F \right)\cap \mathbb{Z}^{\#E(\widetilde{G}\backslash\{s,0\})} \right) , \mbox{ so } \]
	\begin{align*}
		\LD(G)&=\bigcup_{F\subseteq E(G\backslash 0)}{\LD(G,F)}\\ 
		&=\bigcup_{F\subseteq E(G\backslash 0)}{\rho\left(\SolG(F) \right)}\\ 
		&=\bigcup_{F\subseteq E(G\backslash 0)}{\psi\left(\mathcal{F}_{\widetilde{G}\backslash\{s,0\}}\left(\bm{b}_G^F \right)\cap \mathbb{Z}^{\#E(\widetilde{G}\backslash\{s,0\})} \right)} 
	\end{align*}
\end{theorem}
	
	\label{finalflowsproofedition}
	In the proof of Theorem \ref{theoremA} below, it will be more convenient to use an equivalent formulation of Theorem \ref{finalflows}. 
	Instead of considering flows on $\widetilde{G}\backslash\{s,0\}$ with netflow vector $\bm{b}_G^F$, consider flows on $\widetilde{G}\backslash\{s\}$ with netflow vector $(0,\bm{b}_G^F)$, where
		
	\[(0,\bm{b}_G^F)=\left (0,\indegG(1)-\outdegF(1), \ldots , \, \indegG(n)-\outdegF(n), -\#E(G\backslash F) \right ).\] 
	\medskip
	
	Now, we use Theorem \ref{finalflows} to prove Theorem \ref{theoremA}. Before proceeding with the proof, we first recall the relevant notation introduced previously. For a graph $G$ on $[0,n]$, let $\mathcal{R}$ be any reduction tree of $G$ and $\mathcal{T}(G)$ the specific reduction tree whose leaves are encoded by the arrays $\Sol_G(F)$ (constructed in Definition \ref{specialreductiontree}). Recall that $\inseq(\mathcal{R})$ denotes the multiset of left-degree sequences of the leaves of $\mathcal{R}$, and $\LD(G)=\inseq(\mathcal{T}(G))$.
	 
	 \medskip
	 
	\begin{proof}[Proof of Theorem \ref{theoremA}]
		We proceed by induction on the maximal depth of a reduction tree of $G$. For the base case, the only reduction tree possible is the single leaf $G$. For the induction, perform a single reduction on $G$ using fixed edges $r_1=(i,j)$ and $r_2=(j,k)$ with $i<j<k$ to get graphs $G_1$, $G_2$, and $G_3$, with notation as in (\ref{reducing}). Note that we are selecting particular edges $r_1$ and $r_2$ even if there are multiple edges $(i,j)$ or $(j,k)$. Let $r_3$ denote the new edge $(i,k)$ in $G_m$ for each $m \in [3]$. Let $\mathcal{R}(G_m)$ be the reduction tree of $G_m$, $m \in [3]$,  induced from $\mathcal{R}$ by restriction to the node labeled by $G_m$ and all of its descendants. 
		
		By the induction assumption, $\inseq(\mathcal{R}(G_m))$ is exactly  $\LD(G_m)$, so 
		\[\inseq(\mathcal{R})=\bigcup_{m\in [3]}{\inseq(\mathcal{R}(G_m))}=\bigcup_{m\in [3]}{\LD(G_m)}. \]	
		Thus, we need to show that 
		\begin{align}
			\label{T(G)union}
			\LD(G)=\bigcup_{m\in [3]}\LD(G_m)
		\end{align}
		regardless of the choice of $r_1$ and $r_2$. However, if $\rho$ is the map that takes an array to its first column, then Theorem \ref{finalflows} yields the disjoint union decomposition
		\[\LD(G)=\bigcup_{F\subseteq E(G\backslash 0)}\rho\left(\SolG(F) \right).
		\]
		Similarly, for each $m\in[3]$,
		\[\LD(G_m)=\bigcup_{F\subseteq E(G_m\backslash 0)}\rho\left(\SolGm(F) \right)
		\]
		Thus, to prove (\ref{T(G)union}), it suffices to show 
		\begin{align}
			\label{identification}
			\bigcup_{F\subseteq E(G\backslash 0)}\rho\left(\SolG(F) \right)=\bigcup_{m\in [3]}{\bigcup_{F\subseteq E(G_m\backslash 0)}\rho\left(\SolGm(F) \right)}.
		\end{align}
		
		To show (\ref{identification}), to each $F\subseteq E(G\backslash 0)$, we associate a tuple $(F_m)_{m \in I(F, r_1, r_2)}$ with $I(F, r_1, r_2)\subseteq [3]$ and  $F_m\subseteq E(G_m\backslash 0)$, $m\in [3]$,  such that each subset of any $E(G_m\backslash 0)$ is in exactly one tuple and for each $F\subseteq E(G\backslash 0)$,
		\[\rho\left(\SolG(F) \right)=\bigcup_{m \in I(F, r_1, r_2)}{\rho\left(\SolGm(F_m) \right)}. \]
		
		By Theorem $\ref{finalflows}$, we verify the equivalent condition 
		\[\psi\left(\mathcal{F}_{\widetilde{G}\backslash\{s\}}\left(0,\bm{b}_G^F \right)\cap \mathbb{Z}^{\#E(\widetilde{G}\backslash\{s\})} \right)=
			\bigcup_{m \in I(F, r_1, r_2)}{\psi\left(\mathcal{F}_{\widetilde{G}_m\backslash\{s\}}\left(0,\bm{b}_{G_m}^F \right)\cap \mathbb{Z}^{\#E(\widetilde{G}_m\backslash\{s\})} \right)}. \]

		To make the notation more compact, let $H=\widetilde{G}\backslash\{s\}$ and $H_m=\widetilde{G_m}\backslash\{s\}$ for $m\in [3]$. 	We proceed in several cases depending on $F, r_1, r_2$. In each case, the argument is very similar to the proof of Proposition \ref{subdivisionlemma}.
		\medskip
		
\noindent {\it 	I. 	Suppose that $r_1$ is not incident to vertex $0$}. The following four cases deal with this scenario.
	\smallskip

		\noindent \emph{Case 1:} $r_1, r_2\notin F$: Associate to $F$ the tuple $(F_1,F_2)$ with 
		\[F_1=F \mbox{ and } F_2=F.\]
		Let $h$ be an integral flow on $H$ with netflow vector $(0,\bm{b}_G^F)$. For $m\in [3]$, we define integral flows on $H_m$ with netflow $(0,\bm{b}_{G_m}^F)$ having the same image under $\psi$.
		\begin{itemize}
			\item If $h(r_1)\geq h(r_2)$, define $h_1$ on $H_1$ with netflow $\bm{b}_{G_1}^{F_1}$ by 
			\begin{align*}
				h_1(e)=
				\begin{cases}
					h(r_2) &\mbox{ if } e=r_3,\\
					h(r_1)-h(r_2) &\mbox{ if } e=r_1, \\
					h(e) &\mbox{ otherwise.}
				\end{cases}
			\end{align*}
			
			\item If $h(r_1)< h(r_2)$, define $h_2$ on $H_2$ with netflow $\bm{b}_{G_2}^{F_2}$ by 
			\begin{align*}
				h_2(e)=
				\begin{cases}
					h(r_1) &\mbox{ if } e=r_3,\\
					h(r_2)-h(r_1)-1 &\mbox{ if } e=r_2, \\
					h(e) &\mbox{ otherwise.}
				\end{cases}
			\end{align*}
		\end{itemize}
		
		For the inverse map, given integral flows $h_m$ on $H_m$ with netflow $\bm{b}_{G_m}^{F_m}$ for $m\in[2]$, define flows $h^{(m)}$ on $H$ by
		\begin{align*}
			h^{(1)}(e)=
			\begin{cases}
				h_1(r_1)+h_1(r_3) &\mbox{ if } e=r_1,\\
				h_1(r_3) &\mbox{ if } e=r_2, \\
				h_1(e) &\mbox{ otherwise. }
			\end{cases}
			\mbox{ and \hspace{2ex}}
			h^{(2)}(e)=
			\begin{cases}
				h_2(r_3) &\mbox{ if } e=r_1,\\
				h_2(r_2)+h_2(r_3)+1 &\mbox{ if } e=r_2, \\
				h_2(e) &\mbox{ otherwise. }
			\end{cases}
		\end{align*}
		\medskip
		
		\noindent \emph{Case 2:} $r_1\in F,$ $r_2\notin F$: Associate to $F$ the tuple $(F_1,F_2)$ with
		\[F_1= F\backslash \{r_1\}\cup \{r_3\}\mbox{ and } F_2= F\backslash \{r_1\}\cup \{r_3\}. \]
		Use the same maps on flows given in Case 1.
		\medskip

		\noindent \emph{Case 3:} $r_1\notin F,$ $r_2\in F$: Associate to $F$ the tuple $(F_1,F_2,F_3)$ with 
		\[F_1= F\backslash \{r_2\}\cup \{r_1\},\mbox{ } F_2=F,\mbox{ and } F_3=F\backslash \{r_2\}. \] 
		Let $h$ be an integral flow on $H$ with netflow vector $(0,\bm{b}_G^{F})$. For $m\in [3]$, we define integral flows on $H_m$ with netflow $(0,\bm{b}_{G_m}^{F_m})$ having the same image under $\psi$.
		\begin{itemize}
			\item If $h(r_1)> h(r_2)$, define $h_1$ on $H_1$ with netflow $\bm{b}_{G_1}^{F_1}$ by 
			\begin{align*}
				h_1(e)=
				\begin{cases}
					h(r_2) &\mbox{ if } e=r_3,\\
					h(r_1)-h(r_2)-1 &\mbox{ if } e=r_1, \\
					h(e) &\mbox{ otherwise. }
				\end{cases}
			\end{align*}
			
			\item If $h(r_1)< h(r_2)$, define $h_2$ on $H_2$ with netflow $\bm{b}_{G_2}^{F_2}$ by 
			\begin{align*}
				h_2(e)=
				\begin{cases}
					h(r_1) &\mbox{ if } e=r_3,\\
					h(r_2)-h(r_1)-1 &\mbox{ if } e=r_2, \\
					h(e) &\mbox{ otherwise.}
				\end{cases}
			\end{align*}
			
			\item If $h(r_1)= h(r_2)$, define $h_3$ on $H_3$ with netflow $\bm{b}_{G_3}^{F_3}$ by 
			\begin{align*}
				h_3(e)=
				\begin{cases}
					h(r_1) &\mbox{ if } e=r_3,\\
					h(e) &\mbox{ otherwise.}
				\end{cases}
			\end{align*}
		\end{itemize}
		
		Given integral flows $h_m$ on $H_m$ with netflows $\bm{b}_{G_m}^{F_m}$ for $m\in[3]$, construct the inverse map by defining flows $h^{(m)}$ on $H$ for $m\in[3]$. Let $h^{(2)}$ be the same as in Case 1, and define 
		\begin{align*}
			h^{(1)}(e)=
			\begin{cases}
				h_1(r_1)+h_1(r_3)+1 &\mbox{ if } e=r_1,\\
				h_1(r_3) &\mbox{ if } e=r_2, \\
				h_1(e) &\mbox{ otherwise,}
			\end{cases}
			\mbox{ and \hspace{2ex}}
			h^{(3)}(e)=
			\begin{cases}
				h_3(r_3) &\mbox{ if } e=r_1,\\
				h_3(r_3) &\mbox{ if } e=r_2,\\
				h_3(e) &\mbox{ otherwise.}
			\end{cases}
		\end{align*}

		\medskip
		\noindent \emph{Case 4:} $r_1,r_2\in F$: Associate to $F$ the tuple $(F_1,F_2,F_3)$ with
		\[F_1=F\backslash \{r_2\}\cup \{r_3\},\mbox{ } F_2= F\backslash \{r_1\}\cup \{r_3\},\mbox{ and } F_3= F\backslash \{r_1,r_2\}\cup \{r_3\}. \]	
		Use the maps on flows given in Case 3.
		\medskip
		
		A straightforward check shows that every $F\subseteq E(G_m\backslash 0)$ for $m\in [3]$ is reached exactly once by {\it cases 1-4}.\\
		\smallskip

\noindent  {\it II.	Suppose that $r_1$ is incident to vertex $0$}. The following two cases deal with this scenario.
		 		\smallskip
				
	\noindent \emph{Case 1':} $r_2\notin F$: Associate to $F$ the tuple $(F_1,F_2)$ with
		\[F_1=F \mbox{ and } F_2= F. \]
		Use the maps on flows given in Case 1.
	
		\noindent \emph{Case 2':} $r_2\in F$: Associate to $F$ the tuple $(F_2,F_3)$ with 
		\[F_2=F \mbox{ and } F_3=F\backslash \{r_2\}. \]
		Use the maps on flows for $H_2$ and $H_3$ given in Case 3. 
		\smallskip
		
		A straightforward check shows that every $F\subseteq E(G_m\backslash 0)$ for $m\in [3]$ is reached exactly once by {\it cases 1'-2'}.
	\end{proof}

\section*{Acknowledgments} 
We thank Bal\'azs Elek, Alex Fink, and Allen Knutson for inspiring conversations. We are grateful to the anonymous referees for their helpful and detailed feedback, which improved the exposition of the paper.

\bibliography{FlowPolytopesBibliography} 
\bibliographystyle{plain}

\end{document}